\documentclass[11pt,letterpaper]{article}
\usepackage[round]{natbib}

\usepackage{mwe}
\usepackage{subfig}
\usepackage{bbm}
\usepackage{bm}
\usepackage{chngcntr}
\usepackage{multirow}
\usepackage{tikz,pgfplots}
\pgfplotsset{compat=1.5}

\usepackage[colorlinks=true,urlcolor=blue,citecolor=blue,linkcolor=blue,bookmarks=true,bookmarksopen=false]{hyperref}

\usepackage[ruled,linesnumbered,vlined]{algorithm2e}
\usepackage{algorithmic}
\usepackage{booktabs}
\usepackage{hhline}
\usepackage{mathrsfs}

\usepackage{setspace}

\newcommand{\vq}{\mathbf{q}}

\newcommand{\vw}{\mathbf{w}}

\newcommand{\vv}{\mathbf{v}} 
\newcommand{\vz}{\mathbf{z}} 
\newcommand{\vx}{\mathbf{x}} 
\newcommand{\vpi}{\bm{\pi}} 

\renewcommand{\Omega}{\varOmega}
\newcommand{\R}{\mathbbm{R}}

\newcommand{\E}{\mathbbm{E}}

\newcommand{\vtP}{\mathbf{\tilde{P}}}
\newcommand{\vtc}{\mathbf{\tilde{c}}}
\newcommand{\cvar}{\operatorname{CVaR}}

\newcommand{\var}{\operatorname{VaR}}

\definecolor{newgreen}{rgb}{0.0, 0.5, 0.0}
\definecolor{bblue}{HTML}{4F81BD}
\definecolor{rred}{HTML}{C0504D}
\definecolor{ggreen}{HTML}{9BBB59}
\definecolor{ppurple}{HTML}{9F4C7C}

\newcounter{commentcounter}
\long\def\symbolfootnote[#1]#2{\begingroup%
\def\thefootnote{\fnsymbol{footnote}}\footnote[#1]{#2}\endgroup}

\usepackage{sibarticle}
\usepackage{threeparttable}
\usepackage{pdflscape}
\usepackage{mathtools}
\usepackage{tikz}
\usetikzlibrary{arrows}
\usepackage{graphicx}
\usepackage{caption}

\allowdisplaybreaks
\newcommand{\ignore}[1]{}

\DeclareMathOperator*{\argmax}{arg\,max}

\newcommand{\revised}[1]{\textcolor{black}{#1}}
\newcommand{\red}[1]{\textcolor{red}{#1}}
\newcommand{\blue}[1]{\textcolor{blue}{#1}}

\title{\revised{Risk Aversion to Parameter Uncertainty in Markov Decision Processes} with an Application to Slow-Onset Disaster Relief}
\ShortTitle{} \ShortAuthors{\vspace{-0.1cm} } \NumberOfAuthors{1}
\FirstAuthor{ }
\FirstAuthorAddress{\vspace{-1.5cm}\quad }
 \ShortTitle{MDPs under parameter uncertainty} \ShortAuthors{Merakl{\i}, K\"{u}\c{c}\"{u}kyavuz} \NumberOfAuthors{1}
 \FirstAuthor{Merve Merakl{\i}, Simge K\"{u}\c{c}\"{u}kyavuz}
 \FirstAuthorAddress{Industrial Engineering \& Management Sciences, Northwestern  University, Evanston, IL, U.S.A.\\ {\tt merakli@northwestern.edu, simge@northwestern.edu}}


\begin{document}

\maketitle \centerline{\today}

\begin{abstract}
In classical Markov Decision Processes (MDPs), action costs and transition probabilities are assumed to be known, although an accurate estimation of these parameters is often not possible in practice. This study addresses MDPs under cost and transition probability uncertainty and aims to provide a mathematical framework to obtain policies minimizing the risk of high long-term losses due to not knowing the true system parameters. To this end, we utilize the risk measure value-at-risk associated with the expected performance of an MDP model with respect to parameter uncertainty. We provide mixed-integer linear and nonlinear programming formulations and heuristic algorithms for such risk-averse \revised{models of} MDPs under a finite distribution of the uncertain parameters. Our proposed models and solution methods are illustrated on an inventory management problem for humanitarian relief operations during a slow-onset disaster. The results demonstrate the potential of our risk-averse modeling approach for reducing the risk of highly undesirable outcomes in uncertain/risky environments.

\end{abstract}

\section{Introduction}

Markov Decision Processes (MDPs) are effectively used in many applications of sequential decision making in uncertain environments including inventory management, manufacturing, robotics, communication systems, and healthcare, e.g., \cite{puterman2014markov,altman1999constrained,boucherie2017markov}. In an MDP model, the decision makers take an action at specified points in time considering the current state of the system with the aim of minimizing their expected loss (resp., maximizing their expected utility), and depending on the action taken, the system transitions to another state. The evolution of the underlying process is mainly characterized by the action costs (resp., rewards) and transition probabilities between the system states, inducing two types of uncertainty. The \emph{internal uncertainty} stems from the probabilistic behaviour of transitions between states, costs and actions (see, e.g., \citealp{ruszczynski2010risk,bauerle2011markov,xu2011probabilistic,fan2018risk} for studies addressing the risk arising from internal uncertainty). The \emph{parameter uncertainty}, on the other hand, is due to the ambiguities in the parameters representing the costs and transition probabilities.
In classical MDPs, these parameters are assumed to be known; they are usually estimated from historical data or learned from previous experiences. However, in practice, it is usually not possible to obtain a single estimate that fully captures the nature of the uncertainties. The actual performance of the system may significantly differ from the anticipated performance of the MDP model due to the inherent variation in the parameters \citep{mannor2007bias}. Our focus in this study is on the parameter uncertainty in MDPs---decision makers are assumed to be sensitive to the risk associated with the fluctuations in parameters while being risk-neutral to the internal randomness due to state transitions, costs and randomized actions. This setting is especially suitable for applications in which the objective function is aggregated over a number of problem instances, e.g., total inventory cost over various types of supply items.
In such cases, aggregation across multiple instances mitigates the variation due to internal uncertainty, and parameter uncertainty becomes the main source of variation.

A widely used approach to incorporate parameter uncertainty into MDPs is robust optimization.
In the robust modeling framework, the objective is to optimize the worst-case performance over all possible realizations in a given uncertainty set. This approach is appealing in the sense that it requires no prior information on the distribution of costs or transition probabilities and it gives rise to computationally efficient solution algorithms. However, it often leads to conservative results because the focus is on the worst-case system performance, which may be rarely encountered in practice. In the initial studies on robust MDPs, uncertainty is usually described using a polyhedral set, because it leads to tractable solution algorithms \citep{satia1973markovian,white1994markov,givan1997bounded,tewari2007bounded,bagnell2001solving}. The uncertainty sets are later extended for more general definitions with the aim of balancing the conservatism of the solutions and the tractability of the solution algorithms. \citet{nilim2005robust} model the uncertainty in transition probabilities using a set of stochastic matrices satisfying \textit{rectangularity} property, i.e., when there are no correlations between transition probabilities for different states and actions. The authors devise an efficient dynamic programming algorithm for this case. Similarly, \citet{iyengar2005robust} studies robust MDPs under transition probability uncertainty with rectangularity assumption and provides robust value and policy iteration algorithms for \textit{finite-horizon} nonstationary and \textit{infinite-horizon} stationary problem settings. \citet{sinha2016policy} propose a policy iteration algorithm for robust infinite-horizon nonstationary MDPs following the rectangularity assumption. \citet{wiesemann2013robust} relax the rectangularity assumption of \cite{nilim2005robust} and consider a more general class of uncertainty sets in which the assumption of no correlation between transition probabilities is only made for states, not for actions. \citet{mannor2016robust} define a tractable subclass of nonrectangular uncertainty sets, namely $k$-rectangular uncertainty sets, such that the number of possible conditional projections of the uncertainty set is at most $k$. Another alternative modeling approach to balance the conservatism of the solutions is to consider the distributional robustness, where the uncertain parameters are assumed to follow the worst-case distribution from a set of possible distributions described by some general properties such as expectations or moments. Unlike robust MDPs, distributionally robust MDPs incorporate the available---but incomplete---information on the a priori distribution of the uncertain parameters. For relavant studies on distributionally robust MDPs, we refer the reader to \citet{xu2012dist,yu2016distributionally}.

Bayesian approaches to address parameter uncertainty have been receiving increasing attention in the recent literature. This uncertainty model treats unknown parameters as random variables with corresponding probability distributions. Hence, it provides the means to incorporate complete distributional information about the unknown parameters and the attitude of decision makers towards uncertainty and risk, oftentimes at the cost of increasing problem complexity. \citet{steimle2018multimodel} consider a multi-model MDP, where the aim is to find a policy maximizing the weighted sum of expected total rewards over a finite horizon associated with different sets of parameters obtained by different estimation methods. This modeling framework is analogous to an expected value problem considering a finite number of scenarios in the context of stochastic programming, because each set of parameters can be treated as a scenario in which the corresponding scenario probability is set as the normalized weight value. The authors prove the existence of a deterministic optimal policy, show that the problem is NP-hard, and provide a mixed-integer linear programming (MILP) formulation and a heuristic algorithm. In a subsequent work, \citet{steimle2018decomposition} propose a customized branch-and-bound algorithm to solve the multi-model MDPs. \citet{buchholz2018computation} study multi-model MDPs, where each model corresponds to an infinite-horizon MDP. Unlike the finite-horizon case, it is demonstrated that infinite-horizon variant may not have a deterministic optimal policy. The authors propose two nonlinear programming formulations and a heuristic algorithm for the case where randomized policies are allowed, and an MILP model utilizing the dual linear programming formulation of MDPs for the deterministic policy case. 

A potential drawback of previously mentioned modeling frameworks for multi-model MDPs is that the expected value objective ignores the risk arising from the parameter uncertainty. Considering this issue, \citet{xu2009parametric} address reward uncertainty in MDPs with respect to parametric regret. The MDP of interest is a finite-horizon, discounted model with finite state and action spaces. The authors consider two different objectives: minimax regret based on the robust approach and mean-variance trade-off of the regret based on the Bayesian approach. They propose a nonconvex quadratic program for the former objective and a convex quadratic program for the latter. \citet{delage2010percentile} consider reward and transition probability uncertainty separately and propose a chance-constrained model in the form of percentile optimization, which corresponds to the risk measure, \textit{value-at-risk} ($\var$). The authors give a formulation for infinite-horizon MDPs with finite state and action spaces, and stationary policies. They show that the problem is intractable in the general case but can be efficiently solved when the rewards follow a Gaussian distribution or transition probabilities are modeled using independent Dirichlet priors \revised{satisfying rectangularity property}. Alternatively, \citet{chen2012tractable} investigate a class of percentile-based objective functions that are easy to approximate for any probability distribution and \citet{adulyasak2015solving} focus on finding objective functions that are separable over realizations of uncertain parameters under a sampling framework. Our modeling approach for incorporating parameter uncertainty into MDPs is along the same line with \citet{delage2010percentile}, however we allow general probability distributions with finite support. Throughout the manuscript, we use the terms percentile, quantile and $\var$ interchangeably.


In this study, we  study MDPs under cost and transition probability uncertainty. Our aim is to obtain a stationary policy that optimizes the quantile function value, $\var_\alpha$, at a certain confidence level $\alpha$ with respect to parameter uncertainty. 
The $\var_\alpha$ assesses an estimate of the largest potential loss excluding the worst outcomes with at most $1-\alpha$ probability.
Conservatism of solutions in optimization problems involving $\var_\alpha$ can be adjusted using different confidence levels $\alpha$, reflecting decision makers' risk aversion---quantile optimization is equivalent to a robust optimization approach when $\alpha=1$.
In addition, quantile-based performance measures, such as $\var$, are used in many applications in the service industry because of their clear interpretation and correspondence with the service-level requirements, e.g., minimum investment to guarantee 100$\alpha$\% service level \citep{decandia2007dynamo,benoit2009benefits,atakan2017minimizing}. Although $\var$ is nonconvex in general, the main challenge in our case is the combinatorial nature of policies independent from the choice of the risk measure, which makes the problem NP-hard even for the expected value objective \citep{steimle2018multimodel}.
Unlike \citet{delage2010percentile}, who also consider the $\var$ objective \revised{for cost and transition probability uncertainty cases separately under certain rectangularity assumptions}, we assume that action costs and transition probabilities are both uncertain and follow a finite joint distribution \revised{without imposing any independence requirements on system parameters}. This approach directly represents the cases in which a finite set of possible parameter realizations can be obtained based on historical data and/or multiple estimation tools \citep{bertsimas2016robust,steimle2018multimodel}. Moreover, it provides a general framework that can be used on sample approximations of a wide range of probability distributions.
Finite representation of uncertainty also facilitates competitive solution methods for optimization problems incorporating $\var$. Since $\var$ corresponds to a quantile function, optimization problems incorporating $\var$ can be formulated as chance-constrained programs (CCPs), which are known to be NP-hard in general and usually require multidimensional integration. When uncertain parameters follow a finite distribution, the challenges of working with multivariate distributions can be circumvented, and the CCP formulations can be stated as mixed-integer programming (MIP) models by employing big-$M$ inequalities and additional binary variables for each possible realization of uncertainty \citep{luedtke2008sample} (see, e.g., \citealp{kuccukyavuz2012mixing,liu2017intersection,zhao2017polyhedral} for further strengthenings). Majority of the studies on CCPs considers a single-state (i.e., static) decision-making framework, where the uncertainty is revealed only after all required decisions are made. \citet{luedtke2014branch} and \citet{liu2016decomposition} extend the literature for a two-stage decision-making framework such that recourse decisions are allowed in the second stage and provide branch-and-cut algorithms employing mixing inequalities to ensure feasibility/optimality of the second-stage problems. Along the same lines, \citet{zhang2014branch} consider a multi-stage (finite-horizon) setting and propose a branch-and-cut algorithm using continuous mixing inequalities. For an overview of CCPs and related approaches, we refer the reader to \cite{kuccukyavuz2017introduction}.
Motivated by these advances in the CCP literature, \citet{feng2015practical} formulate the $\var$ portfolio optimization problem as a single-stage CCP and provide an MILP reformulation and a branch-and-cut algorithm utilizing the mixing inequalities of \cite{luedtke2014branch}. Their results suggest superiority of the MIP approaches over branch-and-cut based algorithms and emphasize the significance of the big-$M$ terms on computational performance of the MIP formulations for $\var$ optimization. Based on these conclusions, \citet{pavlikov2018optimization} provide a bounding scheme that produces tighter big-$M$ values for the MILP formulation. 

Optimizing $\var$ associated with parameter uncertainty in MDPs, on the other hand, brings out additional challenges due to the combinatorial nature of the decisions and the underlying Markovian system dynamics in an MDP. In this problem setting, the aim is to obtain a single optimal policy (selected at the beginning) minimizing the $\var$ associated with parameter uncertainty in an MDP, which is to be implemented over the entire planning horizon under any realization of uncertainty. Different from the previously mentioned studies on CCPs, the underlying process is assumed to be Markovian and possible actions in each state belong to a finite set. For this purpose, we provide a two-stage CCP formulation capable of modeling the dynamics of a Markov chain for any selected policy considering possible values of uncertain parameters. Note that here the second stage represents the performance of the MDP for a given policy. We additionally propose relaxations and heuristic solution algorithms that can be used for obtaining lower and upper bounds on the optimal objective function value. Although we focus on infinite-horizon MDPs, our results and algorithms can be easily extended for finite-horizon MDPs after small adjustments.

We test our modeling framework and solution algorithms on a humanitarian inventory management problem for relief items required to sustain basic needs of a population affected by a slow-onset disaster, e.g., war, political insurgence, extreme poverty, famine, or drought. Since the progress and impact of a slow-onset disaster generally depend on unpredictable political and/or natural events, the demand for the relief items is highly variable. The supply amounts are also exposed to uncertainty as they mainly rely on voluntary donations. Another critical issue in humanitarian inventory management is the perishability of many relief items such as food and medication. At the beginning of each time period, based on the current inventory level, the decision makers need to determine an additional order quantity to minimize the expected total inventory holding, stock-out and disposal costs considering the expiration dates and the uncertainty in supply and demand. This problem can be modeled as an MDP, where the current inventory level represents the state of the system, and the uncertainty in supply, demand, and inventory is captured by the transition probabilities between different states \citep{ferreira2018inventory}. However, the cost and transition probability parameters of the MDP model are subject to high level of uncertainty because demand and supply rates and shelf life of perishable relief items used in the estimation of these parameters may widely fluctuate.  
The $\var_\alpha$ objective in this setting has a natural interpretation: to find a replenishment strategy that minimizes the budget required to cover the expected total costs considering all possible parameter realizations with at least $\alpha$ probability. 

The rest of this paper is organized as follows. In Section \ref{sec:modelForm}, we formulate the MDP problem \revised{minimizing $\alpha$-quantile} under cost and transition probability uncertainty, provide MIP models and explore characteristics of the optimal policies. Section \ref{sec:ImpDet} presents preprocessing procedures, which can be used for initializing auxiliary parameters of the proposed mathematical models as well as reducing the problem size, and a heuristic solution algorithm. We describe a stochastic inventory management problem for slow-onset disasters in Section \ref{sec:appInv}, which is later used for demonstrating effectiveness of the quantile-optimizing modeling approach and proposed solution methods in Section \ref{sec:compStudy}. \revised{We summarize our contributions and future research directions in Section \ref{sec:Conc}.}

\section{Problem Formulation and Structural Properties} \label{sec:modelForm}

Consider a discrete-time infinite-horizon MDP model with finite state space $\mathcal{H}$ and finite action space $\mathcal{A}$. We define $\tilde{c}_i(a)$ as the immediate expected cost of taking action $a \in \mathcal{A}$ in state $i \in \mathcal{H}$ and $\tilde{P}_{ij}(a)$ as the probability of transitioning from state $i\in \mathcal{H}$ to state $j\in \mathcal{H}$ under action $a \in \mathcal{A}$. The future costs are discounted by $\gamma \in [0,1)$ and the distribution of the initial state is given as $|\mathcal{H}|$-dimensional vector $\vq$. A stationary policy $\vpi=(\pi_1,\dots,\pi_{|\mathcal{H}|})$ refers to a sequence of decision rules $\pi_i$ describing the action strategy for each state $i \in \mathcal{H}$. When the policy is \textit{randomized}, each element of $|\mathcal{A}|$-dimensional vector $\pi_i$ denotes the probability of taking the respective action at each time state $i \in \mathcal{H}$ is encountered. For a \textit{deterministic} policy $\vpi$, on the other hand, $\pi_i$ refers to a unit vector in which only the element corresponding to the action selected in state $i \in \mathcal{H}$ is one. Assuming that the cost and transition probability parameters are nonnegative, stationary and bounded, the expected total discounted cost of the underlying Markov chain for a given policy $\vpi$ and known system parameters $(\vtc,\vtP)$ can be stated as 
\begin{equation*}
    C(\vpi,\vtc,\vtP)=\E_{\vx \in \mathcal{H}}\left(\sum_{t=0}^\infty \gamma^t \tilde{c}_{x_t}(\pi_{x_t})|x_0 \propto \vq,\vpi \right),
\end{equation*}
where $x_0$ and $x_t$ denote the initial state and the state of the system at decision epoch $t>0$, respectively. \revised{The term $\tilde{c}_{x_t}(\pi_{x_t})$ corresponds to the immediate expected cost of taking action $\pi_{x_t}$ in state $x_t$ at time period $t=0,1,\dots$. The expectation is taken based on the given policy $\vpi$ and the initial state $x_0$ following probability distribution $\vq$.} Let $\Pi$ be the set of stationary Markov policies. In an MDP model, the aim is to find a policy $\vpi \in \Pi$ minimizing the expected total discounted cost, i.e.,
\begin{equation}\label{form:cmdp}
    \min\limits_{\vpi \in \Pi}\quad C(\vpi,\vtc,\vtP).
\end{equation}
This problem is known to have a stationary and deterministic optimal solution over the set of all policies, and it can be solved efficiently using several well-known methods such as the value iteration algorithm, the policy iteration algorithm, and linear programming \citep{puterman2014markov}. Using the Bellman equation \citep{bellman2013dynamic},
problem \eqref{form:cmdp} can be alternatively stated as
\begin{subequations}\label{form:mdplp}
\begin{align}
 \min\limits_{\vpi \in \Pi,\ \vv \in \R^{|\mathcal{H}|}}~&  \sum_{i \in \mathcal{H}} q_i v_i  \\
\text{s.t.}~& v_i = \tilde{c}_{i}(\pi_{i}) + \gamma \sum_{j \in \mathcal{H}} \tilde{P}_{ij}(\pi_i) v_j, \quad i \in \mathcal{H},  \label{BellmanEq}
\end{align}
\end{subequations}
where $v_i$ is the expected sum of discounted costs under the selected policy $\vpi$ when starting from state $i \in \mathcal{H}$, which satisfies Bellman optimality condition
\begin{equation} \label{BellmanOpt}
    v_i = \min\limits_{a\in\mathcal{A}} \left\{\tilde{c}_{i}(a) + \gamma \sum_{j \in \mathcal{H}} \tilde{P}_{ij}(a) v_j\right\}, \quad i \in \mathcal{H}.
\end{equation}

The cost vector $\vtc$ and transition probability matrix $\vtP$ in MDP model \eqref{form:cmdp} are assumed to be known. However, in practice, it is usually difficult to predict exact values of these parameters. For example, the prices or  the weekly demand rate for relief items during a slow-onset disaster (e.g., war) are subject to a high level of uncertainty due to various factors. In addition, for cases where rare events may have a tremendous impact---as in the case of disaster management---the decision makers may prefer to incorporate their aversion towards risk into the decision support tools. Motivated by these arguments, in this study, we consider the setting where the elements of the cost vector $\vtc$ and transition probability matrix $\vtP$ are assumed to be random variables instead of known parameters. The decision makers need to determine a policy $\vpi \in \Pi$ in this uncertain environment with the aim of minimizing their risk of realizing a large amount of expected total discounted cost \revised{with respect to the uncertainty in parameters}. In terms of stochastic programming terminology, policy decisions can be thought of as \textit{non-anticipative}, that is, a policy minimizing the risk is selected at the beginning in the presence of uncertainty, and it will be implemented throughout the planning horizon for any realization of parameters $\vtc$ and $\vtP$, the values of which are not revealed. We seek a policy $\vpi \in \Pi$ minimizing the $\alpha$-quantile of the expected total discounted cost with respect to parameter uncertainty, which corresponds to the optimal policy of the \revised{quantile-minimizing} MDP problem 
\begin{subequations}
\begin{align*}
\textup{\revised{(QMDP)}}\quad\min\limits_{y\in \R,\ \vpi \in \Pi}\hspace{0.25cm} & y \\
s.t.\ \ & \mathbb{P}_{\vtc,\vtP}\left(  C(\vpi,\vtc,\vtP) \leq y \right) \geq \alpha. 
\end{align*}
\end{subequations}
Formulation QMDP ensures that the expected total discounted cost for the optimal policy $\vpi^*$ is less than or equal to the optimal objective function value $y^*$ with probability at least $\alpha$ under the distributions of $\vtc$ and $\vtP$. Note that the optimal value $y^*$ corresponds to the $\var$ of the expected total discounted cost for the optimal policy $\vpi^*$ at confidence level $\alpha$. Such formulations optimizing $\var$ ($\alpha$-quantile) are also referred to as quantile or percentile optimization in the literature. \citet{delage2010percentile} consider special cases of the quantile optimization problem QMDP in which cost and transition probability uncertainty are treated separately. The authors assume that the cost parameters follow a Gaussian distribution in the former case, and the transition probabilities in the latter are modeled using Dirichlet priors. Different than their model, we consider the cost and transition probability uncertainty simultaneously without making any assumptions on the distributions of uncertain parameters other than that it can be represented/approximated with a finite discrete distribution.

Using the nominal MDP model in \eqref{form:mdplp}, we obtain an alternative CCP formulation for problem QMDP,
\vspace{-1.0cm}\begin{subequations}\label{form:mdp-sce}
\begin{align}
\min\limits_{y\in \R,\ \vpi \in \Pi,\ \vv \in \R^{|\mathcal{H}|}}\hspace{0.25cm} & y \\
s.t.\ \ & \mathbb{P}_{\vtc,\vtP}\left(\sum_{i \in \mathcal{H}} q_i v_i^{(\vtc,\vtP)} \leq y\right)   \geq \alpha. \label{c3}\\
 & v_i^{(\vtc,\vtP)} = \tilde{c}_{i}(\pi_{i}) + \gamma \sum_{j \in \mathcal{H}} \tilde{P}_{ij}(\pi_i) v_j^{(\vtc,\vtP)} &\ i \in \mathcal{H}.\label{cc4c}
\end{align}
\end{subequations}
\revised{The $|\mathcal{H}|$-dimensional random vector $\vv^{(\vtc,\vtP)}$ is composed of random variables $v_i^{(\vtc,\vtP)}$, whose values depend on the joint realization of random parameters $\vtc$ and $\vtP$ and the selected policy $\vpi$}. In general, such problems with chance constraints are highly challenging to solve since they require computation of the joint distribution function, which usually involves numerical integration in multidimensional spaces
\citep{deak1988multidimensional}.
On the other hand, using a discrete representation of the distribution function obtained by a sampling method significantly reduces the computational complexity and provides reliable approximations to CCPs for a sufficiently large sample size \citep{calafiore2006scenario,luedtke2008sample}. In stochastic optimization, each sample of parameters is referred to as a \emph{scenario}. Note that this approach yields optimal solutions for multi-model MDPs in which parameter uncertainty can be finitely discretized. For example in medical applications for designing optimal treatment and screening protocols, the system state usually represents patient health status and transition probabilities between states can be computed using multiple tools from the clinical literature, which often produce different sets of parameters \citep{steimle2018multimodel}. In this context, the parameters computed using each tool can be treated as a scenario. Similarly, for the humanitarian inventory management problem, each scenario may correspond to a set of parameters computed using a different demand/supply forecasting method. In this case, the $\var$ objective can be interpreted as minimizing the worst-case expected total cost over $\alpha$ fraction of the possible MDP parameters (scenarios).

Another challenge in solving problem \eqref{form:mdp-sce} is the possibility that none of the optimal policies is deterministic. 
\revised{In such cases, finding an optimal policy requires searching over the larger set of randomized policies and the corresponding mathematical models may involve additional nonlinearities. Despite these computational challenges, randomized policies are successfully implemented in various applications of constrained MDPs, such as power management \citep{benini1999policy}, wireless transmission systems \citep{djonin2007mimo}, infrastructure maintenance and rehabilitation \citep{smilowitz2000optimal} and queuing systems with reservation \citep{feinberg1994optimality}.}
As mentioned before, for unconstrained infinite-horizon discounted MDPs with discrete state space, finite action space and known parameters, \revised{an optimal policy (deterministic or randomized) can be efficiently obtained by solving linear programs and} it is ensured that there always exists a deterministic stationary optimal policy. However, this is not necessarily true for MDPs under parameter uncertainty or additional constraints. Even for the expected value objective, the infinite-horizon problem may not have any deterministic optimal policy \citep{buchholz2018computation}, while the existence of a deterministic optimal policy is guaranteed for the finite-horizon problem \citep{steimle2018multimodel}. The problem QMDP, on the other hand, does not necessarily have a deterministic optimal policy for either infinite or finite-horizon cases. The following example demonstrates this observation for the infinite-horizon case, and it can be easily adjusted for a finite-horizon problem with a single period and zero termination costs. 

\begin{example}\label{ex1}
Consider a single-state infinite-horizon MDP with two actions, $a$ and $b$, under two scenarios with equal probabilities. Under scenario 1, $\tilde{c}(a)=0$ and $\tilde{c}(b)=2$, and under scenario 2, $\tilde{c}(a)=2$ and $\tilde{c}(b)=0$. In case a stationary deterministic policy is applied, i.e., either action $a$ or $b$ is chosen, the optimal objective function value of the problem at confidence level $\alpha=0.9$ is $2/(1-\gamma)$. However, if a randomized policy of selecting action $a$ or $b$ with equal probabilities is applied, then the optimal objective function value is $1/(1-\gamma)$, proving that no stationary deterministic policy is optimal.
\end{example}


For classical MDPs with known parameters, the Bellman optimality condition \eqref{BellmanOpt} leads to a linear programming formulation, where policy decisions are implied by the constraints that are satisfied as strict equalities. However, the quantile optimization problem QMDP cannot utilize such implicit representation of the policies because Bellman optimality condition does not necessarily hold when the same policy should be imposed across all scenarios due to \revised{non-anticipativity}. Even though the Bellman optimality condition is no longer valid, value functions in each scenario still need to obey the Bellman equation \eqref{BellmanEq} to correctly represent the dynamics of the underlying Markov chain for any given policy. Using this property, we propose a mixed-integer nonlinear programming (MINLP) formulation for problem \eqref{form:mdp-sce} assuming the existence of a discrete representation for parameter uncertainty as in the following statement:
\begin{itemize}
    \item[A1.] Joint distribution of $\vtc$ and $\vtP$ is represented as a finite set of scenarios $\mathcal{S}=\{1,\ldots,n\}$ with corresponding probabilities $p^1,\ldots,p^n$.
\end{itemize}
Note that our mathematical models and solution algorithms can be easily adjusted for the finite-horizon MDPs with nonstationary/stationary policies by introducing a time dimension on the decisions and/or parameters. Here, we focus on the infinite-horizon case, in which the stationarity assumption is desired for practicality and tractability, and ignore the time indices for brevity. 

\begin{proposition}\label{prop:mipform} 
Assuming a finite representation of parameter uncertainty as described in \textup{A1}, problem \eqref{form:mdp-sce} can be restated as the following deterministic equivalent formulation,
\begin{subequations}\label{form:mdp-mip-r}
\begin{align}
\textup{(QMDP-R)}\quad\min\hspace{0.25cm} & y \label{obj1}\\
s.t.\ \ & \sum_{a \in \mathcal{A}} w_{ia} = 1, \quad i \in \mathcal{H},\label{ccc1}\\
&\sum_{s \in \mathcal{S}} z^s p^s \geq \alpha,\label{ccc2}\\
& \sum_{i \in \mathcal{H}} q_i v_i^s \leq y + (1-z^s) M,\quad s \in \mathcal{S}, \label{ccc3} \\
& v_i^s \geq \sum_{a \in \mathcal{A}} \tilde{c}^s_i(a)  w_{ia} + \gamma \sum_{a \in \mathcal{A}} \sum_{j \in \mathcal{H}} \tilde{P}_{ij}^s(a) v_j^s  w_{ia}, \quad i \in \mathcal{H},\ s \in \mathcal{S}, \label{c8} \\
& z^s \in \{0,1\},\quad s \in \mathcal{S}, \label{c4}\\
& w_{ia} \in [0,1], \quad i \in \mathcal{H},\ a \in \mathcal{A},\label{ccc4-r}
\end{align}
\end{subequations}
where $M$ represents a large number such that constraint \eqref{ccc3} for $s \in \mathcal{S}$ is redundant if $z^s=0$.
\end{proposition}

\begin{proof}
Constraint \eqref{ccc1} combined with the domain constraint \eqref{ccc4-r} ensures a probability distribution on the action space for each state, which is effective for all stages. Hence, $w_{ia}$ corresponds to the probability that action $a$ is taken in state $i$ under the optimal policy. Note that the variables $\vw$ are required for enforcing a single stationary policy across all scenarios. Constraints \eqref{ccc1} and \eqref{ccc4-r} also guarantee that constraint \eqref{c8} is equivalent to the Bellman equations in \eqref{cc4c} for the policy determined by $\vw$, and variable $v_i^s$ corresponds to the value of $v_i^{(\vtc,\vtP)}$ in \eqref{form:mdp-sce} for parameter realizations in scenario $s\in\mathcal{S}$. Finally, by definition of $M$, $z^s$ denotes a binary variable equal to 1 if scenario $s$ satisfies $\sum_{i \in \mathcal{H}} q_i v_i^s \leq y$ and consequently, constraints \eqref{ccc2}--\eqref{ccc3} and \eqref{c4} are equivalent to the chance constraint \eqref{c3}.
\end{proof}


The formulation QMDP-R is nonlinear and nonconvex in general due to constraints \eqref{c8}, which contain the bilinear terms $v_j^s w_{ia}$ for $s \in \mathcal{S}$, $i \in \mathcal{H}$, $a \in \mathcal{A}$. For such MINLP models, most of the existing solution algorithms only guarantee local optimality of the resulting solutions. To obtain a lower bound on the globally optimal objective function value, we approximate each bilinear term $x_{ija}^s=v_j^s w_{ia}$ for $i,j \in \mathcal{H},\ a\in \mathcal{A},\ s\in \mathcal{S}$ in \eqref{c8} by the following linear inequalities using McCormick envelopes \citep{mccormick1976computability}
\begin{subequations} \label{form:mc1}
\begin{align}
  \ell_j^s w_{ia}  \leq x_{ija}^s \leq &\ u_j^s w_{ia}, \\
  v_j^s- (1-w_{ia}) u_j^s \leq   x_{ija}^s \leq &\ v_j^s- (1-w_{ia}) \ell_j^s,\\
  \ell_j^s   \leq v_{j}^s \leq &\ u_j^s , \label{cons:vfBnds}
\end{align}
\end{subequations}
where $\ell_j^s$ and $u_j^s$ are lower and upper bounds on the value of $v_j^s$, respectively. Then, nonlinear constraint \eqref{c8} can be replaced by
\begin{equation}\label{form:mc2}
     v_i^s \geq \sum_{a \in \mathcal{A}} \tilde{c}^s_i(a)  w_{ia} + \gamma \sum_{a \in \mathcal{A}} \sum_{j \in \mathcal{H}} \tilde{P}_{ij}^s(a) x_{ija}^s, \quad i \in \mathcal{H},\ s \in \mathcal{S},
\end{equation}
which provides a linear relaxation of QMDP-R in a lifted space.  

Despite the possibility that all optimal policies are randomized, implementation of deterministic policies may be preferred over randomized policies in certain application areas. These include the cases in which the decision makers are prone to making errors \citep{chen2002dynamic}, and the cases that making randomized decisions raises ethical concerns as in the health care systems \citep{steimle2018multimodel} and humanitarian relief operations. Considering these cases, we make an additional assumption that
\begin{itemize}
    \item[A2.] The policy space is restricted to the set of stationary deterministic policies, denoted as $\Pi_D$, i.e., set $\Pi$ in QMDP is replaced by $\Pi_D$,
\end{itemize}
and propose an MINLP formulation that aims to minimize the $\alpha$-quantile value over all stationary deterministic policies in set $\Pi_D$.

\begin{proposition}\label{prop:mipform-d} 
Under assumptions \textup{A1}--\textup{A2}, problem \eqref{form:mdp-sce} can be restated as the following deterministic equivalent formulation,
\begin{subequations}
\begin{align}
\textup{(QMDP-D)}\quad\min\hspace{0.25cm} & y \nonumber\\
s.t.\ \ & \eqref{ccc1}-\eqref{c4}, \nonumber\\
& w_{ia} \in \{0,1\}, \quad i \in \mathcal{H},\ a \in \mathcal{A}, \nonumber
\end{align}
\end{subequations}
where $M$ represents a large number such that constraint \eqref{ccc3} for $s \in \mathcal{S}$ is redundant if $z^s=0$.
\end{proposition}
The proof easily follows from Proposition \ref{prop:mipform} and the characterization of deterministic policies.

Similar to QMDP-R, Formulation QMDP-D is also nonlinear and nonconvex due to the bilinear terms in constraint \eqref{c8}, but in this case, binary representation of deterministic policies can be further utilized to obtain MILP reformulations. When policy vector $\vw$ is binary, the McCormick envelopes \eqref{form:mc1} correspond to an exact linearization of the bilinear terms $x_{ija}^s=v_j^s w_{ia}$ for $i,j \in \mathcal{H},\ a\in \mathcal{A},\ s\in \mathcal{S}$. Hence, the MILP obtained by replacing \eqref{c8} with \eqref{form:mc2} and adding constraints \eqref{form:mc1} is equivalent to QMDP-D. This McCormick reformulation requires the addition of $\mathcal{O}(|\mathcal{S}||\mathcal{A}||\mathcal{H}|^2)$ variables and constraints. An alternative reformulation of QMDP-D as an MILP of smaller size by a factor of $\mathcal{O}(|\mathcal{H}|)$ can be achieved by replacing \eqref{c8} in QMDP-D with linear inequalities
\begin{equation}
v_i^s \geq \tilde{c}^s_i(a) + \gamma \sum_{j \in \mathcal{H}} \tilde{P}_{ij}^s(a) v_j^s  - (1-w_{ia}) M_{is}, \quad i \in \mathcal{H},\ a \in \mathcal{A},\ s \in \mathcal{S}. \label{c9}
\end{equation}
Constraint \eqref{c9} assures that the multi-scenario Bellman equations \eqref{cc4c} in the CCP model \eqref{form:mdp-sce} are satisfied for the selected scenarios, i.e., $z^s=1$, and the policy represented by the state-action pairs ($i,a$) such that $w_{ia}=1$. The big-$M$ term $M_{is}$ for any $i \in \mathcal{H},\ s \in \mathcal{S}$ denotes a large number sufficient to make constraint \eqref{c9} redundant when $w_{ia}=0$ for $a \in \mathcal{A}$. Note that replacing \eqref{c8} with \eqref{c9} does not provide an exact reformulation of QMDP-R.

\ignore{
In addition, even though the MILP reformulation of QMDP-D with McCormick envelopes requires more variables and constraints, it is also likely to provide tighter bounds than the reformulation with constraint \eqref{c9}, because the big-$M$ terms are known to cause weak linear programming relaxations as also demonstrated by our computational experiments in Section \ref{sec:CompSolTim}.
}



\section{Implementation Details} \label{sec:ImpDet}
In this section, we propose preprocessing methods to initialize problem parameters, which may also be used for reducing the problem size. We also provide a heuristic solution algorithm. Note that the proposed methods are applicable for both QMDP-R and QMDP-D, but here we consider on QMDP-R due to its generality.  

\subsection{Preprocessing} \label{sec:prepro}

The formulations for problems QMDP-R and QMDP-D presented in the previous section require determination of the auxiliary terms $M$, $M_{is}$, $\ell_i^s$ and $u_i^s$ for $i \in \mathcal{H},\ s \in \mathcal{S}$.
In what follows, we provide preprocessing procedures to prespecify these terms and to narrow down the solution space before solving the original problem.
Existing solution algorithms for MDPs usually suffer from the size of state and action spaces, often referred to as the curse of dimensionality. In our case, the computational complexity is additionally amplified with the number of scenarios and the combinatorial nature of the quantile calculation. Hence it is worthwhile to search for methods that reduce the size of the problem. 

Let $y^*$ be the optimal objective function value of problem QMDP-R. First, we relax the requirement that the same policy should be selected over all scenarios, and use monotonicity property of the $\var$ function to find bounds on the value of $y^*$. We denote $\bar b$ as a random variable representing the maximum expected total discounted cost of the relaxed MDP model. The realization of $\bar b$ under scenario $s \in \mathcal{S}$ can be computed by solving the following linear programming problem
\begin{subequations}\label{maxProbS}
\begin{align}
\bar{b}^s=\min\hspace{0.25cm} & \sum_{i \in \mathcal{H}} q_i v_i  \\
s.t.\ \ & v_i \geq \tilde{c}^s_i(a) + \gamma \sum_{j \in \mathcal{H}} \tilde{P}_{ij}^s(a) v_j, \quad i \in \mathcal{H},\ a \in \mathcal{A}.  
\end{align}
\end{subequations}
Then we order the realizations of $\bar{b}$ under each scenario as
$\bar{b}^{s_1}\leq \bar{b}^{s_2}\leq \dots \leq \bar{b}^{s_{\revised{n}}},$
and let $k_1 \in \{1,\dots,n\}$ be the order of the scenario such that $\sum_{t=1}^{k_1} p^{s_t} \geq \alpha$ and $\sum_{t=1}^{k_1-1} p^{s_t} < \alpha$, and $b_u=\bar{b}^{s_{k_1}}$. Note that $b_u$ corresponds to the $\var$ of random variable $\bar{b}$ at confidence level $\alpha$. Hence $b_u$ provides an upper bound on $y^*$, i.e., $b_u\geq y^*$, since $\bar{b}^s \geq C(\vpi,\vtc^s,\vtP^{s})$ for all $s\in \mathcal{S}$, $\vpi \in \Pi$. 

Similarly, let $\underline{b}$ be a random variable representing the minimum expected total discounted cost of the MDP with relaxed policy selection requirements, where realizations of $\underline{b}$ under each scenario $s \in \mathcal{S}$ can be computed by solving the linear programs
\begin{subequations}\label{minProbS}
\begin{align}
\underline{b}^s=\max\hspace{0.25cm} & \sum_{i \in \mathcal{H}} q_i v_i  \\
s.t.\ \ & v_i \leq \tilde{c}^s_i(a) + \gamma \sum_{j \in \mathcal{H}} \tilde{P}_{ij}^s(a) v_j, \quad i \in \mathcal{H},\ a \in \mathcal{A}.  
\end{align}
\end{subequations}
Let $b_l$ correspond to the $\var$ of random variable $\underline{b}$ at confidence level $\alpha$. The value of $b_l$ provides a lower bound on $y^*$, i.e., $b_l\leq y^*$ as $\underline{b}^s \leq C(\vpi,\vtc^s,\vtP^{s})$ for all $s\in \mathcal{S}$, $\vpi \in \Pi$ by definition.

Problems \eqref{maxProbS} and \eqref{minProbS} can be efficiently solved using a policy iteration or value iteration algorithm in polynomial time. Using these bounds, we can conclude that any scenario $s$, whose lower bound is greater than the upper bound on the quantile value, i.e., $\underline{b}^s>b_u$, cannot satisfy the term inside chance constraint \eqref{c3} in the optimal policy, hence we can set $z^s=0$. Similarly, any scenario $s\in \mathcal{S}$ with an upper bound value smaller than the lower bound on the quantile function value, i.e., $\bar{b}^s<b_l$ must satisfy the term inside chance constraint \eqref{c3} in the optimal solution. Therefore, the corresponding scenario variable can be set as $z^s=1$. This result can be used for reducing the number of $\vz$ variables. Furthermore, we can add the constraints $b_l \leq y\leq b_u$
into our mathematical model. As previously mentioned, validity of these bounds follows from monotonicity property of the $\var$ function. 
Additionally, the inequalities 
\begin{equation}\label{cons:lb}
    y \geq \underline{b}^s z^s,\quad s \in \mathcal{S}
\end{equation}
are valid due to constraints \eqref{ccc3} and the fact that $\underline{b}^s \leq C(\vpi,\vtc^s,\vtP^{s})$ for all $s\in \mathcal{S}$, $\vpi \in \Pi$. Note that \citet{kuccukyavuz2016cut} propose similar bounding and scenario elimination ideas and demonstrate their effectiveness in the context of multivariate $\cvar$-constrained optimization  (see also, \citealp{Liu2017,Noyan2019}).

Let $\bar{v}_i^s$ and $\underline{v}_i^s$ be the optimal values of variable $v_i$ in problem \eqref{maxProbS} and \eqref{minProbS}, respectively, for state $i \in\mathcal{H}$ and scenario $s\in\mathcal{S}$. Then, we can set the lower and upper bounds on variable $v_i^s$ in \eqref{form:mc1} as $\ell_i^s=\underline{v}_i^s$ and $u_i^s=\bar{v}_i^s$, respectively. This result follows from the fixed-point and contraction properties of the value functions for classical MDPs. The information obtained by solving problems \eqref{maxProbS} and \eqref{minProbS} can also be used to obtain tighter values of the big-$M$ terms in constraints \eqref{ccc3} and \eqref{c9}. Clearly, we can set \revised{$M=b_u-b_l$} in \eqref{ccc3} because its value is bounded by $M \geq \max\limits_{\vv:\eqref{ccc1}-\eqref{ccc4-r}} \sum_{i \in \mathcal{H}} q_i v_i^s - \min\limits_{y:\eqref{ccc1}-\eqref{ccc4-r}} y $ for $s\in \mathcal{S} \text{ such that } z^s =1.$
Moreover, \citet{steimle2018multimodel} show that the big-$M$ term in constraint \eqref{c9} can be set as $M_{is} = \bar{v}_i^s-\underline{v}_i^s$ for all $i\in \mathcal{H},\ s \in\mathcal{S}$. 

\subsection{Obtaining a Feasible Solution}\label{sec:MDPheuristic}

Our preliminary computational experiments suggest that the upper bound $b_u$ described in the previous section can be further improved by finding a feasible solution to problem QMDP-R. Here, we propose a polynomial time heuristic algorithm (Algorithm \ref{alg2}), which benefits from the connection between a substructure of our problem with the robust MDPs to attain a feasible policy effectively.

\begin{algorithm}[ht]
\setstretch{1.20}
Given distinct cost and transition matrices $\{\vtc,\vtP\}_{s\in\mathcal{S}}$ with corresponding probabilities $\{p\}_{s\in\mathcal{S}}$, set $\hat{\vz} \gets \mathbf{0}$\;
\For{each scenario $s \in \mathcal{S}$}{Solve problem \eqref{minProbS} to obtain its optimal objective function value $\underline{b}^s$\;}
Compute $\var_\alpha(\underline{b})$ and find a maximal subset of scenarios $\bar{\mathcal{S}}\subseteq \mathcal{S}$ such that $\underline{b}^s \leq \var_\alpha(\underline{b})$ for each scenario $s \in \bar{\mathcal{S}}$ and $\sum\limits_{s\in\bar{\mathcal{S}}\setminus\{s' \} } p^s < \alpha$ for all $s'\in\bar{\mathcal{S}} $\;
\For{each scenario $s \in \bar{\mathcal{S}}$}{Set $\hat{z}^s \gets 1$\;}
Return \emph{robustPolicySelection($\hat{\vz}$)}\;
\caption{\emph{initialSolution()} \label{alg2}}
\end{algorithm}

Algorithm \ref{alg2} follows two sequential phases: scenario selection and policy selection. In the scenario selection phase, we decide on which scenarios will be enforced to satisfy the chance constraint, i.e., $\hat{z}^s = 1$. As in the previous subsection, we relax the requirement that the selected policy should be the same over all scenarios, and solve problems \eqref{minProbS} independently to obtain the optimal objective function value $\underline{b}^s$ corresponding to each scenario $s \in \mathcal{S}$. Then, we set $\hat{z}^s = 1$ for each scenario $s$ in a maximal subset $\bar{\mathcal{S}}\subseteq \mathcal{S}$ such that $\underline{b}^s \leq \var_\alpha(\underline{b})$ for each scenario $s \in \bar{\mathcal{S}}$ and $\sum\limits_{s\in\bar{\mathcal{S}}\setminus\{s' \} } p^s < \alpha$ for all $s'\in\bar{\mathcal{S}} $.

In the policy selection phase provided in Algorithm \ref{alg3}, we use the scenarios selected in the first phase to obtain a feasible policy and the corresponding quantile value. In other words, we let $\mathcal{S}(\vz)=\bar{\mathcal{S}}$ in Line 1 of Algorithm \ref{alg3} and solve a relaxation of the robust MDP problem 
\begin{equation*}
\text{rMDP($\vz$):} \quad \min\limits_{\vpi\in\Pi} \max\limits_{s\in \mathcal{S}(\vz)}  C(\vpi,\vtc^s,\vtP^{s}), 
\end{equation*}
where $\mathcal{S}(\vz):=\{s\in\mathcal{S}|\ z^s=1\}$, to find a feasible policy for QMDP-R. For any scenario vector $\vz$, rMDP($\vz$) can be seen as an adversarial game, where decision maker selects a stationary policy at the very beginning, and the system evolves based on the worst possible scenario for the selected policy afterwards. An important observation is that for any vector $\vz$ satisfying constraints \eqref{ccc2} and \eqref{c4}, the optimal policy obtained by solving rMDP($\vz$) and its optimal objective function value correspond to a feasible solution of QMDP-R for policy variables $\vw$ and quantile variable $y$, respectively. To see this, we reformulate QMDP-R using the definition of $C(\vpi,\vtc^s,\vtP^{s})$ as
\begin{subequations}
\begin{align*}
\min\limits_{y\in \R,\ \vpi \in \Pi}\hspace{0.25cm}  & y \\
s.t.\ \ & y \geq C(\vpi,\vtc^s,\vtP^{s}) - (1-z^s)M, \quad \quad s \in \mathcal{S} ,\\
& \eqref{ccc2}, \eqref{c4},
\end{align*}
\end{subequations}
or equivalently $\min\limits_{\vz:\eqref{ccc2}, \eqref{c4},\ \vpi \in \Pi}\hspace{0.25cm}   \max\limits_{s\in \mathcal{S}(\vz)}  C(\vpi,\vtc^s,\vtP^{s}).$

Note that the substructure rMDP($\vz$) is a robust MDP, where the uncertainty in the transition matrix is coupled across the time horizon and the state space, i.e., a single realization of the cost and transition matrices is selected randomly at the beginning and it holds for all decision epochs and states of the system. Problem rMDP($\vz$) is known to be NP-hard even when the parameters are allowed to follow a different scenario for each state and only the scenario considered at each time a state is encountered should be consistent over time \citep{iyengar2005robust}. 
 Therefore, in Algorithm \ref{alg3}, using the scenario vector $\hat \vz$ obtained in the scenario selection phase, we find a policy by solving a relaxed version of $\text{rMDP($\hat \vz$)}$, which we describe next.

\begin{algorithm}[ht]
\setstretch{1.20}
Given a small tolerance parameter $\epsilon>0$, set some positive $\hat{\vv}_1\in\R^{|\mathcal{H}|}$, $\mathcal{S}(\vz)\gets \{s\in\mathcal{S}|\ z^s=1\}$ and $t \gets 1$\;
\For{each state $i \in \mathcal{H}$}{ \For{each action $a \in \mathcal{A}$}{  Set $\sigma_{ia} \gets \max\limits_{s \in \mathcal{S}(\vz) }\left\{c_i^s(a) + \gamma  \sum_{j \in \mathcal{H}} \tilde{P}_{ij}^s(a) \hat{v}_t(j)  \right\} $ \;} Compute $\hat{\vv}_{t+1}(i) \gets \min\limits_{a\in \mathcal{A}}\ \sigma_{ia}  $    \;}
\eIf{$||\hat{\vv}_{t+1}-\hat{\vv}_t ||< \frac{(1-\gamma)\epsilon}{\gamma} $}{Go to line 10\;}{Set $t\gets t+1$ and go to line 2\;}
Return $\vpi$ such that $\pi_i \in \argmax\limits_{a \in \mathcal{A}}\ \sigma_{ia}$\;
\caption{\emph{robustPolicySelection($\vz$)} \label{alg3}}
\end{algorithm}

Similar to \citet{nilim2005robust}, we consider a relaxed variant of $\text{rMDP($\hat \vz$)}$ by assuming that the costs and transition probabilities  are independent over different states and actions, and these parameters are allowed to be time-varying. This setting can be seen as a sequential game, where at each time step, an action is taken by the decision maker and accordingly the system generates a cost and makes a transition based on the worst possible parameter realization in scenario set $\mathcal{S}(\hat \vz)$ for the current state and action, in an iterative fashion. 
Note that \citet{nilim2005robust} consider the case of transition matrix uncertainty only and prove that the optimal policy in this setting obeys a set of optimality conditions, which can be solved using a robust value iteration algorithm. 
Here, we further generalize their algorithm for the case of uncertainty in both cost and transition matrices to find a policy that performs well in terms of the VaR objective, but is not necessarily optimal. For the vector $\hat{\vz}$ computed in the first phase, we construct the set $\mathcal{S}(\hat{\vz}):=\{s\in\mathcal{S}|\ \hat{z}^s=1\}$ and solve the following equations
\begin{equation*}
v_i= \min\limits_{a\in \mathcal{A}} \max\limits_{s \in \mathcal{S}(\vz) }\left\{c_i^s(a) + \gamma  \sum_{j \in \mathcal{H}} \tilde{P}_{ij}^s(a) v_j  \right\}, \quad i \in \mathcal{H} 
\end{equation*}
through a variant of the value iteration algorithm as presented in Algorithm \ref{alg3}.
Note that the policy obtained by Algorithm \ref{alg3} and the corresponding value functions provide upper bounds on the results of the robust problem $\text{rMDP($\hat \vz$)}$. Moreover, due to line 10, Algorithm \ref{alg3} is guaranteed to terminate with a deterministic policy, hence the obtained solution is also feasible for problem QMDP-D.

\section{An Application to Inventory Management in Long-Term Humanitarian Relief Operations} \label{sec:appInv}

Long-term humanitarian relief operations, alternatively referred to as continuous aid operations, are vital to sustain daily basic needs of a population affected by a slow-onset disaster including war (Syria, Afghanistan, Iraq), political insurrection (Syria), famine (Yemen, South Sudan, Somalia), drought (Ethiopia) and extreme poverty (Niger, Liberia). Unlike the sudden-onset disasters (e.g., earthquakes, hurricanes, terrorist attacks), they require delivery of materials such as food, water and medical supplies to satisfy a chronic need over a long period of time. Since the progress of a slow-onset disaster usually presents irregularities in terms of its scale and location, the demand is highly uncertain. In addition, the supply levels are also exposed to uncertainty as they mainly depend on donations from multiple resources. More than 90\% of the people affected by slow-onset disasters lives in developing countries, hence the required relief items are usually outsourced from resources in various locations around the world \citep{rottkemper2012transshipment}. Another important consideration in long-term humanitarian operations is the perishability of the relief items. Many items needed during and after a disaster, e.g., food and medication, have a limited shelf life. In addition to the possible uncertainty in the initial shelf life, the remaining shelf life upon arrival is also affected by the uncertainty in the lead times due to unknown location of the donations. Hence, high level of uncertainty inherent in the supply chain makes the system prone to unwanted shortages and disposals.
 To prevent interruptions, the policy makers may interfere through different actions such as campaigns and advertisements that provide additional relief items. 

In this section, we consider an inventory management model for a single perishable item in long-term humanitarian relief operations proposed by \citet{ferreira2018inventory}. The authors formulate the problem as an infinite-horizon MDP with finite state and action spaces, where the states represent possible levels of inventory. The aim is to minimize the long-term average expected cost. Their model assumes that both demand and supply is uncertain following Poisson distributions with known demand and supply rates, respectively, and the expiration time of the supply items is deterministic. These parameters are generally obtained using various forecasting methods. However, due to multiple sources of uncertainty in the context of long-term humanitarian operations, the forecasted values may be erroneous, which may consequently affect the performance of the resulting policy. To handle the parameter uncertainty, different from \cite{ferreira2018inventory}, we assume that demand and supply rates and the expiration time probabilistically take value from a finite set of scenarios $\mathcal{S}$ and the objective in our model is to minimize the $\alpha$-quantile of the total discounted expected total cost with respect to the uncertainty in these parameters.

In what follows, we elaborate on the components and assumptions of the MDP model based on \cite{ferreira2018inventory}.

\noindent\underline{\emph{States:}} The state of the system is represented by the number of available items in the inventory at the beginning of each decision epoch that will not expire before delivery. It is assumed that the inventory has a capacity $K$, e.g., the state space $\mathcal{H} = \{0,1,\dots,K\}$. 

\noindent\underline{\emph{Actions:}} At the beginning of each decision epoch, the decision maker takes an action $a$ from a finite set $\mathcal{A}$ that provides an additional $n_a$ number of relief items. The set of actions taken for each possible state of the system that minimizes the objective function is referred to as the optimal policy.

\noindent\underline{\emph{Transition Probabilities:}} Under each scenario $s \in \mathcal{S}$, the demand ($D$) and supply ($U$) for the item follow Poisson distribution with probability mass functions $f_d^s$ and $f_u^s$ of rates $\mu_d^s$ and $\mu_u^s$, respectively, and the expiration time takes value $t_e^s$. \ignore{We assume that the demand rate is at least as large as the supply rate. }Let $\Delta_{\min}$ and $\Delta_{\max}$ be the minimum and maximum possible difference between the supply and demand at any decision epoch considering a $100(1-\epsilon)$\% confidence level for a small $\epsilon>0$. Note that the minimum possible demand and supply amounts are assumed as zero so that $\Delta_{\min}$ and $\Delta_{\max}$ represent the negative of maximum possible demand and the maximum possible supply, respectively. Then, ignoring the perishability of the item and the actions taken by the decision maker, the probability of transitioning from inventory level $i\in \mathcal{H}$ to $j\in \mathcal{H}$ under scenario $s \in \mathcal{S}$ is
\begin{equation*}
  \hat{P}_{ij}^s=\begin{cases}
    \sum\limits_{\Delta=\Delta_{\min}}^{j-i} \sum\limits_{D=0}^{-\Delta_{\min}} f_d^s(D)f_u^s(D+\Delta), & \text{if $j=0$},\\
    \sum\limits_{\Delta=j-i}^{\Delta_{\max}} \sum\limits_{D=0}^{-\Delta_{\min}} f_d^s(D)f_u^s(D+\Delta), & \text{if $j=K$},\\
    \sum\limits_{D=0}^{-\Delta_{\min}} f_d^s(D)f_u^s(D+j-i), & \text{otherwise}.\\
  \end{cases}
\end{equation*}
Instead of keeping track of the number of items that expire in each period, the model is simplified by assuming that the items procured in each decision epoch have the same expiration time. As a result, at the beginning of each decision epoch, the expiration probability for the whole batch of newly arriving items is computed as the probability of not being able to consume all available items before the expiration time. Note that since the demand follows Poisson distribution with rate $\mu_d^s$, the probability of consuming $k$ items before time $t_e^s$ follows Erlang distribution with parameters $\mu_d^s$ and $k$, which can be stated as $(1-f_e^s(t_e^s,k))$. 
Incorporating perishability of the items based on this simplification, the probability of a transition from inventory level $i\in \mathcal{H}$ to inventory level $j\in \mathcal{H}$ under scenario $s \in \mathcal{S}$ when action $a \in \mathcal{A}$ is taken can be computed as
\begin{equation}\label{prob:trans}
  \tilde{P}_{ij}^s(a)=\begin{cases}
     (1-f_e^s(t_e^s,j))\ \hat{P}_{i+n_a,j}^s, & \text{if $j>i+n_a$},\\
    \sum\limits_{j'=i+n_a+1}^{\min(K,i+n_a+\Delta_{\max})}  f_e^s(t_e^s,j')\ \hat{P}_{i+n_a,j'}^s + \hat{P}_{i+n_a,j}^s, & \text{if $j=i+n_a$},\\
    \hat{P}_{i+n_a,j}^s, & \text{otherwise}.\\
  \end{cases}
\end{equation}
Equation \eqref{prob:trans} follows from the fact that if an arriving batch of items is known to expire, the new arrivals are immediately used to fulfill the demand with priority over the existing inventory (the existing inventory can be used later since it is guaranteed not to expire by the previous assumption), and the remaining items in the batch are disposed. Hence, the case $j>i+n_a$ can occur only if the new items do not expire. Similarly, $j=i+n_a$ implies either that the incoming supply is at least as much as the demand in the current period and the supply surplus is disposed, or demand is equal to the sum of supply and additional items acquired by the action taken. 

\noindent\underline{\emph{Cost function:}} The expected cost of taking action $a \in \mathcal{A}$ at inventory level $i\in \mathcal{H}$ under scenario $s \in \mathcal{S}$, stated as $\tilde c_i^s(a)$, consists of three main components: the inventory holding cost, the expected shortage cost and the expected disposal cost. Assuming a unit disposal cost of $\tilde d$, the total expected disposal cost when action $a \in \mathcal{A}$ is taken at inventory level $i\in \mathcal{H}$ under scenario $s \in \mathcal{S}$ is
\begin{align*}
DC(i,a,s) =& \sum_{\Delta=1}^{\min(K-i-n_a,\Delta_{\max})} f_e^s(t_e^s,i+n_a+\Delta) \hat{P}_{i+n_a,i+n_a+\Delta}^s\ \Delta\ \tilde d \\
& \hspace{1.8cm}  + \sum_{j=K}^{K+n_a+\Delta_{\max}} \sum_{D=0}^{-\Delta_{\min}} f_d^s(D)f_u^s(D+j-i-n_a)\ (j-K)\ \tilde d. 
\end{align*}
The first summation term is due to the expired items, while the second term is  \revised{because the arrivals that cannot be stored in inventory due to capacity limit $K$ are disposed}. Similarly, the shortage costs are stated as 
\begin{align*}
SC(i,a,s) =&  \sum_{j=i+n_a+\Delta_{\min}}^{-1} \sum_{D=0}^{-\Delta_{\min}} f_d^s(D)f_u^s(D+j-i-n_a)\ (-j)\ \tilde u, \quad i\in \mathcal{H},\ a \in \mathcal{A},\ s \in \mathcal{S},
\end{align*}
where $\tilde u$ is the unit shortage cost. Assuming that the unit inventory holding cost is $h$, the total expected cost of taking action $a \in \mathcal{A}$ at inventory level $i\in \mathcal{H}$ under scenario $s \in \mathcal{S}$ can be computed using 
$ \tilde c_i^s(a) = hi + DC(i,a,s)+ SC(i,a,s), \quad i\in \mathcal{H},\ a \in \mathcal{A},\ s \in \mathcal{S}. $ 
Note that the cost structure may be different depending on the characteristics of the environment in consideration. The MDP formulation provides decision makers the flexibility to use even nonconvex cost functions. Our methodology only requires the cost terms to be bounded. 

\section{Computational Study} \label{sec:compStudy}

In this section, we conduct computational experiments on the long-term humanitarian relief operations inventory management problem described in Section \ref{sec:appInv} to examine the effects of incorporating risk aversion towards parameter uncertainty into MDP models and to compare the efficiency of different solution approaches. The problem instances are generated based on the experiments provided in \cite{ferreira2018inventory} considering the inventory of a blood center, which collects and distributes blood packs to support humanitarian relief operations. We suppose that the inventory replenishment decisions are made on a weekly basis, and the weekly demand and supply rates for the blood packs take values \revised{uniformly} in the intervals $[30,130]$ and $[20,80]$, respectively. After donation, each blood pack has a shelf life of six weeks, however unknown lead times may affect the remaining shelf life at the time of arrival to the blood center. Hence, the shelf life is assumed to be \revised{uniform} on the interval $[1,6]$. In case of need, the center may procure additional supply of blood packs by sending up to $V\in \{2,3,4\}$ blood collection vehicles to distant areas so that $\mathcal{A} = \{0,1,\dots,V\}$. Each vehicle collects additional 20 blood packs at the expense of incurring a certain cost. We use the cost parameters for additional procurement, inventory holding costs, disposal costs and shortage costs given in \cite{ferreira2018inventory}. Note that state action costs are scaled by 1000 in our computations.
We additionally assume that the blood center has a capacity of $K\in\{50,100,150\}$ units and the blood packs are in batches of 10 units so that $\mathcal{H}= \{0,1,\dots,\lfloor \frac{K}{10}\rfloor\}$. Different than \cite{ferreira2018inventory}, we generate random instances with \revised{equiprobable scenarios $\mathcal{S}=\{1,\dots,n\}$ for $n \in \{50,100,250,500,1000\}$}, where the parameters of shelf life and demand and supply rates for each scenario randomly take value on their respective intervals stated above. The distribution on the initial state is assumed to be uniform as well.

All experiments are performed using single thread of a Linux server with  Intel Haswell E5-2680 processor at 2.5 GHz and 128 GB of RAM using Python 3.6 \revised{combined with} Gurobi Optimizer 7.5.1 \revised{for MILP formulations, and Pyomo 5.5 \citep{pyomo} and Bonmin 1.8 \citep{bonmin} for MINLP formulations}. The time limit for each instance on Gurobi is set to 3600 seconds and we use the default settings for the MIP gap and feasibility tolerances. The results are obtained for the instances with \revised{$\alpha\in \{0.90,0.95,1\}$} and $\gamma=0.99$. Unless otherwise stated, each reported value corresponds to the average of \revised{five} replications.


\subsection{The impact of deterministic policies}

We first investigate how the quantile value \revised{and the computational performance of solution methods are} affected by the choice of narrowing down the policy space to deterministic policies by comparing the solutions obtained by solving QMDP-R, which includes randomized policies, with the solutions of QMDP-D considering only deterministic policies. \revised{Note that implementation of a randomized policy in the context of humanitarian inventory management may require ordering a random quantity of critical humanitarian supplies (e.g., blood packs) even if the current inventory level is extremely low. Deterministic policies, on the other hand, lead to consistent order quantities at each inventory level and therefore are more desirable in such humanitarian settings due to their interpretability---in the sense that they are easier to explain and justify. Hence, our goal in this section is mainly to determine this cost of interpretability and also to investigate computational aspects of obtaining optimal randomized policies for broader contexts, where implementation of randomized policies does not raise ethical/practical issues, such as power management, wireless transmission systems, infrastructure maintenance and rehabilitation and queuing systems with reservation as previously mentioned.}

Because finding a globally optimal solution for \revised{the nonconvex formulation} QMDP-R is computationally challenging, we use the open-source MINLP solver Bonmin, which provides local optima to nonconvex problems, in combination with Python-based, open-source optimization modeling language Pyomo. The solver is warm-started with the policy obtained by solving Algorithm \ref{alg2}. Since this approach only guarantees local optimality, we additionally obtain lower bounds by solving the McCormick relaxation of QMDP-R as described in Section \ref{sec:modelForm}. \revised{Our preliminary experiments demonstrate that the McCormick relaxation provides up to 1.26\% tighter lower bound values than the lower bound computed by considering each scenario separately. Hence, we opt for the McCormick relaxation for the analysis in this section, where our main concern is to have as tight a lower bound as possible. On the other hand, we use the lower bounds obtained by considering each scenario separately in the preprocessing phase, because this is much more efficient than the McCormick relaxation of QMDP-R, which is a mixed-integer program (due to $z^s$ variables) requiring significant solution times when the sample size is large (see Table 1).
The McCormick relaxation of QMDP-R and deterministic policy problem QMDP-D are solved under the best settings as detailed in Section \ref{sec:CompSolTim}.} 

\begin{table}[htbp]
  \centering
  \caption{Comparison of deterministic and randomized policies}
  \resizebox{16cm}{!}{%
    \begin{tabular}{rrr|rr|rr|rrr}
    \multicolumn{3}{c|}{Instance}     & \multicolumn{2}{c|}{Deterministic}      & \multicolumn{2}{c|}{Randomized -- NL}       & \multicolumn{3}{c}{\revised{Lower bound -- MC}}  \\ \hhline{|---|--|--|---|}
    $\mathcal{|S|}$     & $\mathcal{|H|}$     & $\mathcal{|A|}$     & Time (s) & Gap (\%)  & Time (s) & Diff. (\%)  & Time (s) & Gap (\%) & Diff. (\%) \\ \hhline{|---|--|--|---|}
50	&	6	&	3	&	0.53	&	0.00	&	9.42	&	0.11	&	0.17	&	0.00	&	1.43	\\
50	&	6	&	4	&	1.54	&	0.00	&	10.31	&	0.46	&	0.21	&	0.00	&	6.74	\\
50	&	6	&	5	&	6.96	&	0.00	&	75.52	&	0.40	&	0.26	&	0.00	&	14.80	\\
50	&	11	&	3	&	5.01	&	0.00	&	31.38	&	0.00	&	0.45	&	0.00	&	0.41	\\
50	&	11	&	4	&	48.97	&	0.00	&	52.73	&	0.10	&	0.64	&	0.00	&	1.90	\\
50	&	11	&	5	&	1371.54	&	0.00	&	432.63	&	0.21	&	2.31	&	0.00	&	6.85	\\
50	&	16	&	3	&	48.41	&	0.00	&	146.35	&	0.03	&	0.84	&	0.00	&	0.30	\\
50	&	16	&	4	&	$773.00^{\blue{1}}$	&	0.26	&	307.02	&	0.05	&	1.22	&	0.00	&	0.77	\\
50	&	16	&	5	&	$2883.24^{\blue{4}}$	&	2.72	&	$2502.51^{\red{1}}$	&	-0.47	&	6.32	&	0.00	&	3.33	\\
100	&	6	&	3	&	1.11	&	0.00	&	18.29	&	0.13	&	0.91	&	0.00	&	0.69	\\
100	&	6	&	4	&	6.76	&	0.00	&	83.42	&	0.30	&	4.86	&	0.00	&	8.11	\\
100	&	6	&	5	&	109.59	&	0.00	&	130.38	&	0.65	&	2.02	&	0.00	&	19.91	\\
100	&	11	&	3	&	1.75	&	0.00	&	71.51	&	0.02	&	5.61	&	0.00	&	0.12	\\
100	&	11	&	4	&	331.41	&	0.00	&	246.10	&	0.18	&	22.41	&	0.00	&	2.57	\\
100	&	11	&	5	&	$3483.96^{\blue{4}}$	&	7.61	&	4940.98	&	0.07	&	8.34	&	0.00	&	10.60	\\
100	&	16	&	3	&	4.07	&	0.00	&	296.91	&	0.02	&	35.50	&	0.00	&	0.07	\\
100	&	16	&	4	&	$1762.18^{\blue{2}}$	&	1.04	&	1789.07	&	0.07	&	20.23	&	0.00	&	1.70	\\
100	&	16	&	5	&	$3600.01^{\blue{5}}$	&	5.07	&	2359.51	&	0.05	&	23.11	&	0.00	&	5.53	\\
250	&	6	&	3	&	3.40	&	0.00	&	134.41	&	0.16	&	3.70	&	0.00	&	1.38	\\
250	&	6	&	4	&	26.24	&	0.00	&	3636.85	&	0.10	&	36.90	&	0.00	&	6.89	\\
250	&	6	&	5	&	317.29	&	0.00	&	$5177.90^{\blue{1}}$	&	-0.73	&	51.30	&	0.00	&	13.01	\\
250	&	11	&	3	&	13.97	&	0.00	&	588.75	&	0.02	&	182.71	&	0.00	&	0.26	\\
250	&	11	&	4	&	$1558.46^{\blue{1}}$	&	0.49	&	$10329.68^{\blue{2}}$	&	-0.03	&	74.69	&	0.00	&	2.56	\\
250	&	11	&	5	&	$3600.03^{\blue{5}}$	&	5.19	&	$3713.75^{\blue{4}}$	&	-1.11	&	126.33	&	0.00	&	6.26	\\
250	&	16	&	3	&	42.39	&	0.01	&	$1700.00^{\blue{1}}$	&	-0.02	&	5.03	&	0.00	&	0.15	\\
250	&	16	&	4	&	$3035.12^{\blue{4}}$	&	1.42	&	$3913.63^{\blue{5}}$	&	-0.87	&	$2532.96^{\blue{2}}$	&	1.64	&	0.70	\\
250	&	16	&	5	&	$3600.01^{\blue{5}}$	&	5.34	&	$3887.01^{\blue{4}}$	&	-1.56	&	$3587.00^{\blue{4}}$	&	14.49	&	1.37	\\
500	&	6	&	3	&	10.04	&	0.00	&	884.64	&	0.09	&	1.21	&	0.00	&	1.04	\\
500	&	6	&	4	&	228.39	&	0.00	&	$2123.58^{\blue{2}}$	&	-0.28	&	200.24	&	0.00	&	7.41	\\
500	&	6	&	5	&	$3137.55^{\blue{2}}$	&	2.42	&	$3767.50^{\blue{5}}$	&	-1.18	&	116.72	&	0.00	&	16.40	\\
500	&	11	&	3	&	30.83	&	0.00	&	1305.78	&	0.02	&	4.42	&	0.00	&	0.24	\\
500	&	11	&	4	&	$3077.46^{\blue{3}}$	&	0.74	&	$3692.29^{\blue{5}}$	&	-0.15	&	$2150.46^{\blue{1}}$	&	0.33	&	2.41	\\
500	&	11	&	5	&	$3600.02^{\blue{5}}$	&	7.48	&	$4166.01^{\blue{5}}$	&	-1.12	&	$3600.08^{\blue{5}}$	&	36.17	&	-0.57	\\
500	&	16	&	3	&	264.96	&	0.00	&	$2872.86^{\blue{2}}$	&	-0.03	&	$3600.19^{\blue{5}}$	&	0.27	&	-0.07	\\
500	&	16	&	4	&	$3600.02^{\blue{5}}$	&	1.36	&	$6368.03^{\blue{5}}$	&	-0.86	&	$3600.13^{\blue{5}}$	&	37.67	&	-0.97	\\
500	&	16	&	5	&	$3600.04^{\blue{5}}$	&	6.72	&	$5627.73^{\blue{5}}$	&	-1.26	&	$3600.17^{\blue{5}}$	&	81.17	&	-1.61	\\
1000	&	6	&	3	&	74.47	&	0.00	&	$1864.43^{\blue{1}}$	&	-0.07	&	1496.29	&	0.00	&	1.56	\\
1000	&	6	&	4	&	$1930.16^{\blue{1}}$	&	0.50	&	$3691.58^{\blue{5}}$	&	-0.45	&	$2163.89^{\blue{2}}$	&	5.35	&	7.10	\\
1000	&	6	&	5	&	$3600.07^{\blue{5}}$	&	8.44	&	$3682.36^{\blue{5}}$	&	-2.22	&	$3600.05^{\blue{5}}$	&	30.16	&	8.18	\\
1000	&	11	&	3	&	210.63	&	0.00	&	$2960.04^{\blue{2}}$	&	0.02	&	$3600.23^{\blue{5}}$	&	14.97	&	-0.20	\\
1000	&	11	&	4	&	$3600.03^{\blue{5}}$	&	3.07	&	$3807.47^{\blue{5}}$	&	0.04	&	$3600.12^{\blue{5}}$	&	81.79	&	-1.18	\\
1000	&	11	&	5	&	$3600.02^{\blue{5}}$	&	6.95	&	$4296.76^{\blue{5}}$	&	-0.30	&	$3600.19^{\blue{5}}$	&	100.00	&	-2.71	\\
1000	&	16	&	3	&	$2003.15^{\blue{2}}$	&	0.24	&	$17223.97^{\blue{5}}$	&	-0.04	&	$3600.24^{\blue{5}}$	&	28.95	&	-0.09	\\
1000	&	16	&	4	&	$3600.09^{\blue{5}}$	&	3.58	&	$9645.38^{\blue{5}}$	&	-0.09	&	$3600.35^{\blue{5}}$	&	85.56	&	-0.25	\\
1000	&	16	&	5	&	$3600.11^{\blue{5}}$	&	6.30	&	$11432.46^{\blue{5}}$	&	-0.06	&	$3600.37^{\blue{5}}$	&	100.00	&	0.42		 \\ \hhline{|---|--|--|---|}
\multicolumn{3}{l|}{Average} & $1459.67^{\blue{1.76}}$ & 1.71 & $3022.20^{\blue{1.87},\red{0.02}}$ & -0.21 & $1166.03^{\blue{1.42}}$ & 13.75 & 3.46 \\
\multicolumn{3}{l|}{Maximum} & $3600.11^{\blue{5}}$ & 8.44& $17223.97
^{\blue{5},\red{1}}$ & 0.65 & $3600.37^{\blue{5}}$ & 100.00 & 19.91

    \end{tabular}%
    }
  \label{tab:deterVSrand}%
\end{table}%

Table \ref{tab:deterVSrand} compares the results \revised{at confidence level $\alpha=0.95$} for deterministic policies (QMDP-D) reported under column ``Deterministic'' with the results of the case with randomized policies: the locally optimal solution obtained by solving QMDP-R using the MINLP solver Bonmin under column ``Randomized -- NL'' and its lower bound provided by the McCormick relaxation under the column ``Lower bound -- MC''. \revised{Each value under column ``Time (s)'' refers to the solution time of the corresponding problem. It can be seen that Gurobi terminates within the time limit of one hour for the instances of QMDP-D and the McCormick relaxation of QMDP-R, whereas Bonmin takes longer than the time limit to terminate for some QMDP-R instances due to its non-preemptive internal processes. The blue (resp., red) superscripts correspond to the number of instances that the solver exceeded the time limit with (resp., without) a feasible solution. We additionally report the optimality gap at the time of termination for linear formulations QMDP-D and the McCormick relaxation under the columns ``Gap (\%)''. Note that the nonlinear solver does not provide an optimality gap for QMDP-R because it can only prove local optimality.} The percentage difference between the optimal objective value considering only the deterministic policies, denoted by ``$o_d$'', and the objective value of each solution approach allowing for randomized policies, denoted by ``$o_r$'', is reported as Diff. (\%) = $100 \times \frac{o_d-o_r}{o_d}$ under the column of the corresponding solution method. 

\revised{The results in Table \ref{tab:deterVSrand} show that solving the nonlinear formulation QMDP-R provides solutions which perform at least as well as the optimal deterministic policies for all instances that terminate with a feasible solution within one hour. The gain achieved by permitting randomized policies is at most 0.65\%. However, considering the time limit, the best deterministic policy outperforms the best randomized policy by 0.21\% on average.}
Additionally, the results for the McCormick relaxation, reported under column ``Lower bound -- MC'', indicate that, for our problem instances, the maximum possible gain in the optimal quantile value obtained by allowing for randomized policies is bounded by \revised{3.46\%} on  average and \revised{19.91\%} at maximum. We emphasize that, unlike ``Randomized -- NL'' that produces feasible solutions to QMDP-R, the solutions of ``Lower bound -- MC'' provide only theoretical upper bounds on the actual improvement that can be gained by allowing randomized policies. The results of ``Lower bound -- MC'' may not be possible to achieve in practice because the reformulation using McCormick envelopes corresponds to a relaxation of the original problem QMDP-R. The lower bounds provided by ``Lower bound -- MC'' also indicate that for \revised{149} of \revised{225} problem instances, the reduction in risk ($\var_\alpha$) due to policy randomization is guaranteed not to be more than 3\% with respect to the objective value of the optimal deterministic policy. \revised{Moreover, a deterministic policy of good quality can usually be obtained in reasonable solution times (1459.67 seconds on average), comparable to the solution times of the QMDP-R (3022.20 seconds on average) and its McCormick relaxation (1166.03 seconds on average).  Furthermore,  ``Lower bound -- MC'' fails to provide valid lower bounds for the problem for large instances that reach the time limit as evidenced by the negative entries under the Diff column.   }

Motivated by these results \revised{and given our humanitarian problem context}, in the following sections, we focus on the problem QMDP-D, which only considers deterministic policies.

\subsection{The impact of parameter uncertainty}

Next we examine the effect of incorporating parameter uncertainty into the MDP model in terms of the value gained by using the available stochastic information and the potential loss due to not having perfect information on the true realization of random parameters.

In Table \ref{tab:valueStoc}, we compare the optimal objective function value of the quantile optimization problem QMDP-D, \revised{denoted as} OPT, with two benchmark cases \revised{for $\alpha\in\{0.90,0.95,1\}$}. \revised{Each blue superscript represents the number of replications that terminate with a feasible solution due to time limit for the corresponding $\alpha$-quantile minimization problem.} The first benchmark assumes that the decision maker waits until observing the actual parameter realizations and makes a decision for each scenario independently. This approach does not provide a feasible solution to the original problem because it may produce distinct policies for each scenario. Clearly the quantile value in this case corresponds to a lower bound on the OPT since it can be stated as $\text{LB} = \var_{\alpha}(\underline{b})$, where the realization under scenario $s\in \mathcal{S}$ is $\underline{b}^s = \min\limits_{\vpi\in\Pi_D} C(\vpi,\vtc^s,\vtP^s)$. Using this value, we compute the value of perfect information on MDP parameters as $\text{VPI} = \text{OPT} - \text{LB}$ and its percentage value as $\text{\%VPI} = 100 \times \frac{\text{OPT} - \text{LB}}{\text{OPT}}$.  Column  \revised{\%VPI}  presented in Table \ref{tab:valueStoc} indicates that the losses in the quantile function value due to not knowing the true parameter realizations \revised{are 0.58\%, 4.77\% and 8.95\% on average, and at most 4.19\%, 20.03\% and 25.16\%  for $\alpha= 1, 0.95, 0.90$}, respectively, for our problem instances. Therefore, it is worthwhile to use additional information on the uncertain parameters whenever possible. \revised{Incorporating additional information is especially valuable for decision makers whose main concern is to ensure a robust system performance against parameter uncertainty at lower confidence levels.}

\begin{table}[ht]
  \centering
  \caption{Value of perfect information and value of stochastic solution of the quantile optimization problem}
  \resizebox{13cm}{!}{%
    \begin{tabular}{ccc|rrr|rrr}
    \multicolumn{3}{c|}{Instance}    &  \multicolumn{3}{c|}{\%VPI}  & \multicolumn{3}{c}{\%VSS} \\ \hhline{|---|---|---|}
    $\mathcal{|S|}$     & $\mathcal{|H|}$     & $\mathcal{|A|}$   & $\alpha = 1$ & $\alpha = 0.95$ & $\alpha = 0.90$ & $\alpha = 1$ & $\alpha = 0.95$ & $\alpha = 0.90$  \\ \hhline{|---|---|---|}
50	&	6	&	3	&	0.48	&	1.43	&	2.92	&	0.91	&	4.30	&	3.08	\\
50	&	6	&	4	&	1.59	&	6.74	&	12.91	&	7.66	&	9.87	&	13.59	\\
50	&	6	&	5	&	2.09	&	14.80	&	17.30	&	13.07	&	22.94	&	23.86	\\
50	&	11	&	3	&	0.08	&	0.41	&	1.11	&	0.41	&	0.97	&	1.67	\\
50	&	11	&	4	&	0.28	&	1.90	&	6.18	&	1.43	&	3.27	&	4.61	\\
50	&	11	&	5	&	0.45	&	6.92	&	$8.89^{\blue{4}}$	&	2.34	&	5.97	&	$5.67^{\blue{4}}$	\\
50	&	16	&	3	&	0.05	&	0.30	&	1.20	&	0.25	&	0.84	&	1.50	\\
50	&	16	&	4	&	0.18	&	$0.77^{\blue{1}}$	&	$4.07^{\blue{4}}$	&	1.04	&	$1.78^{\blue{1}}$	&	$2.31^{\blue{4}}$	\\
50	&	16	&	5	&	0.30	&	$3.44^{\blue{4}}$	&	$7.27^{\blue{5}}$	&	1.99	&	$3.02^{\blue{4}}$	&	$4.90^{\blue{5}}$	\\
100	&	6	&	3	&	0.18	&	0.69	&	3.61	&	0.44	&	0.98	&	1.17	\\
100	&	6	&	4	&	2.22	&	8.11	&	13.77	&	3.93	&	4.71	&	3.89	\\
100	&	6	&	5	&	4.19	&	20.03	&	25.16	&	14.40	&	15.93	&	14.63	\\
100	&	11	&	3	&	0.04	&	0.12	&	1.36	&	0.20	&	0.38	&	0.92	\\
100	&	11	&	4	&	0.35	&	2.61	&	$8.00^{\blue{2}}$	&	1.72	&	2.60	&	$4.57^{\blue{2}}$	\\
100	&	11	&	5	&	0.86	&	$10.70^{\blue{4}}$	&	$17.25^{\blue{5}}$	&	2.84	&	$7.10^{\blue{4}}$	&	$5.82^{\blue{5}}$	\\
100	&	16	&	3	&	0.03	&	0.07	&	$1.19^{\blue{1}}$	&	0.17	&	0.23	&	$0.84^{\blue{1}}$	\\
100	&	16	&	4	&	0.19	&	$1.79^{\blue{2}}$	&	$5.23^{\blue{5}}$	&	1.01	&	$1.72^{\blue{2}}$	&	$5.00^{\blue{5}}$	\\
100	&	16	&	5	&	0.44	&	$5.72^{\blue{5}}$	&	$11.44^{\blue{5}}$	&	2.52	&	$6.05^{\blue{5}}$	&	$3.81^{\blue{5}}$	\\
250	&	6	&	3	&	0.08	&	1.38	&	3.35	&	0.32	&	0.96	&	1.46	\\
250	&	6	&	4	&	1.61	&	6.98	&	11.75	&	3.42	&	3.41	&	3.47	\\
250	&	6	&	5	&	1.04	&	13.13	&	24.77	&	11.42	&	15.81	&	15.03	\\
250	&	11	&	3	&	0.03	&	0.27	&	0.89	&	0.15	&	0.45	&	1.74	\\
250	&	11	&	4	&	0.37	&	$2.71^{\blue{1}}$	&	$5.90^{\blue{4}}$	&	1.03	&	$1.28^{\blue{1}}$	&	$2.29^{\blue{4}}$	\\
250	&	11	&	5	&	0.30	&	$6.50^{\blue{5}}$	&	$13.55^{\blue{5}}$	&	1.70	&	$4.00^{\blue{5}}$	&	$6.08^{\blue{5}}$	\\
250	&	16	&	3	&	0.02	&	0.15	&	$0.70^{\blue{1}}$	&	0.16	&	0.34	&	$1.12^{\blue{1}}$	\\
250	&	16	&	4	&	$0.38^{\blue{1}}$	&	$1.92^{\blue{4}}$	&	$4.04^{\blue{5}}$	&	$0.52^{\blue{1}}$	&	$0.99^{\blue{4}}$	&	$2.14^{\blue{5}}$	\\
250	&	16	&	5	&	0.23	&	$5.81^{\blue{5}}$	&	$9.69^{\blue{5}}$	&	1.96	&	$2.62^{\blue{5}}$	&	$4.32^{\blue{5}}$	\\
500	&	6	&	3	&	0.10	&	1.04	&	3.00	&	0.31	&	0.65	&	1.14	\\
500	&	6	&	4	&	1.84	&	7.61	&	11.99	&	2.86	&	2.72	&	2.92	\\
500	&	6	&	5	&	0.87	&	$16.97^{\blue{2}}$	&	$23.45^{\blue{4}}$	&	7.12	&	$14.78^{\blue{2}}$	&	$15.70^{\blue{4}}$	\\
500	&	11	&	3	&	0.04	&	0.24	&	1.10	&	0.16	&	0.21	&	0.48	\\
500	&	11	&	4	&	0.38	&	$2.64^{\blue{3}}$	&	$5.71^{\blue{5}}$	&	1.03	&	$1.72^{\blue{3}}$	&	$2.09^{\blue{5}}$	\\
500	&	11	&	5	&	0.29	&	$7.79^{\blue{5}}$	&	$16.86^{\blue{5}}$	&	1.44	&	$3.74^{\blue{5}}$	&	$3.02^{\blue{5}}$	\\
500	&	16	&	3	&	0.02	&	0.20	&	$0.44^{\blue{1}}$	&	0.16	&	0.20	&	$0.39^{\blue{1}}$	\\
500	&	16	&	4	&	0.25	&	$1.60^{\blue{5}}$	&	$6.20^{\blue{5}}$	&	0.49	&	$1.13^{\blue{5}}$	&	$-0.44^{\blue{5}}$	\\
500	&	16	&	5	&	$0.22^{\blue{1}}$	&	$6.84^{\blue{5}}$	&	$15.60^{\blue{5}}$	&	$0.91^{\blue{1}}$	&	$1.51^{\blue{5}}$	&	$-1.53^{\blue{5}}$	\\
1000	&	6	&	3	&	0.15	&	1.63	&	3.50	&	0.26	&	0.45	&	0.68	\\
1000	&	6	&	4	&	1.57	&	$7.66^{\blue{1}}$	&	12.92	&	1.89	&	$2.44^{\blue{1}}$	&	2.35	\\
1000	&	6	&	5	&	0.67	&	$14.17^{\blue{5}}$	&	$23.70^{\blue{5}}$	&	3.12	&	$13.44^{\blue{5}}$	&	$13.19^{\blue{5}}$	\\
1000	&	11	&	3	&	0.05	&	0.57	&	$1.26^{\blue{3}}$	&	0.14	&	0.20	&	$0.32^{\blue{3}}$	\\
1000	&	11	&	4	&	0.55	&	$3.19^{\blue{5}}$	&	$7.27^{\blue{5}}$	&	0.93	&	$1.17^{\blue{5}}$	&	$1.31^{\blue{5}}$	\\
1000	&	11	&	5	&	0.24	&	$6.89^{\blue{5}}$	&	$18.49^{\blue{5}}$	&	1.31	&	$2.45^{\blue{5}}$	&	$-0.08^{\blue{5}}$	\\
1000	&	16	&	3	&	0.04	&	$0.44^{\blue{2}}$	&	$2.54^{\blue{5}}$	&	0.15	&	$0.26^{\blue{2}}$	&	$-1.01^{\blue{5}}$	\\
1000	&	16	&	4	&	$0.50^{\blue{2}}$	&	$3.36^{\blue{5}}$	&	$7.70^{\blue{5}}$	&	$1.06^{\blue{2}}$	&	$0.13^{\blue{5}}$	&	$-2.26^{\blue{5}}$	\\
1000	&	16	&	5	&	$0.19^{\blue{1}}$	&	$6.34^{\blue{5}}$	&	$17.44^{\blue{5}}$	&	$1.36^{\blue{1}}$	&	$-0.07^{\blue{5}}$	&	$-5.03^{\blue{5}}$	\\
 \hhline{|---|---|---|}
    \multicolumn{3}{l|}{Average}       & $0.58^{\blue{0.11}}$  & $4.77^{\blue{1.76}}$  & $8.95^{\blue{2.42}}$  & $2.26^{\blue{0.11}}$  & $3.77^{\blue{1.76}}$  & $3.84^{\blue{2.42}}$ \\
    \multicolumn{3}{l|}{Maximum}   & $4.19^{\blue{2}}$ & $20.03^{\blue{5}}$ & $25.16^{\blue{5}}$ & $14.40^{\blue{2}}$ & $22.94^{\blue{5}}$ & $23.86^{\blue{5}}$ \\
    \end{tabular}%
    }
  \label{tab:valueStoc}%
\end{table}%

The second benchmark considers the quantile function value corresponding to a policy obtained by solving a single MDP with expected parameter values, referred as the mean value problem in stochastic programming. This provides a heuristic approach to solve QMDP, and the corresponding \revised{$\alpha$-quantile}, denoted as MV, can be computed by treating the expected total cost for the policy obtained by solving the mean value problem under each scenario as a realization of the corresponding random variable, and it provides an upper bound on OPT. Based on MV, the value of incorporating stochasticity of parameters into the MDP model can be measured as $\text{VSS} = \text{MV}-\text{OPT}$ and $\text{\%VSS} = 100 \times \frac{\text{MV} - \text{OPT}}{\text{MV}}$. \revised{The negative \%VSS values in Table \ref{tab:valueStoc} are due to the instances that terminate with a suboptimal OPT value because of the time limit.} The results show that by incorporating uncertainty in the parameters, the $\alpha$-quantile value can be improved by \revised{2.26\%, 3.77\% and 3.84\%} on average with a maximum improvement of \revised{14.40\%, 22.94\% and 23.86\% at $\alpha= 1, 0.95, 0.90 $, respectively.} Hence, it is possible to achieve significant reductions in risk by incorporating parameter uncertainty into MDPs, and thereby reducing the possibility of undesirable outcomes. \revised{Additionally, just as in the case of VPI, the value of considering parameter uncertainty is usually higher for risk-averse decision making at lower confidence levels.}    

\subsection{The impact of incorporating risk aversion}

Given the advantages of incorporating parameter uncertainty into MDPs, next, we analyze how the performance of a policy would be affected by decision makers' attitude towards risk. For rarely occurring events such as disasters, it is important to ensure that the selected policy performs well even under undesirable realizations of random parameters. \revised{Our goal in this section is to demonstrate potential reductions in risk achieved by solving the proposed risk-averse model instead of the expected value problem introduced by \cite{steimle2018multimodel} for such cases where the decision makers are concerned about their performance with respect to the worst $\alpha$-quantile. 
 We point out that an appropriate modeling approach and confidence level $\alpha$, which reflect decision makers' preferences and attitude towards parameter uncertainty, can be determined by performing a sensitivity analysis using our proposed model at different $\alpha$ values as well as the expected value problem.} 

Table \ref{tab:expATvar} reports the $\alpha$-quantile value for the policy obtained by solving the expected value problem given in \citet{steimle2018multimodel} \revised{under a time limit of one hour}, namely \revised{``E-VaR''}, and compares it to \revised{the best $\alpha$-quantile value obtained within one hour}, denoted as ``$\var^*$'' \revised{at confidence levels $\alpha = 0.90,0.95,1$ for small instances with 50--100 scenarios and large instances with 250--500 scenarios.}
The percentage deviation of the quantile value for the expected value policy from the \revised{best} quantile value is given under the \revised{corresponding} column \revised{``E-VaR $- \text{VaR}^*$ (\%)''} as $100 \times \frac{\text{E-VaR}- \var^*}{\text{E-VaR}}$. \revised{Each blue superscript denotes the number of replications for which the corresponding expected value problem or the $\alpha$-quantile problem terminates with a feasible solution due to time limit. Note that negative values are reported for the instances with larger number of scenarios for which the expected value policy performs better than the VaR-minimizing policy due to the increased computational complexity of the quantile optimization problem.} 
The results in Table \ref{tab:expATvar} show that the policy obtained by solving the expected value problem may perform worse than the $\var$-optimizing policy \revised{by up to 3.96\%, 5.30\% and 5.79\% for smaller problem instances and up to 3.54\%, 3.05\% and 3.02\% for large instances in terms of the $\alpha$-quantile value, for $\alpha = 1, 0.95, 0.90$, respectively.} 
Thus, minimizing the VaR objective instead of the expected value may be beneficial for taking extreme outcomes into consideration and the proposed risk-averse modeling framework may provide significant savings in the quantile value.

\begin{table}[h]
  \centering
  \caption{Performance of the optimal policy of the expected value problem at the $\alpha$-quantile}
    \begin{tabular}{ccc|rrr|ccc|rrr}
    \multicolumn{6}{c|}{Small Instances} & \multicolumn{6}{c}{Large Instances} \\ \hhline{|---|---|---|---|}
    \multicolumn{3}{c|}{Instance} & \multicolumn{3}{c|}{E-VaR $- \text{VaR}^*$ (\%)} & \multicolumn{3}{c|}{Instance} & \multicolumn{3}{c}{E-VaR $- \text{VaR}^*$ (\%)} \\ \hhline{|---|---|---|---|}
    $\mathcal{|S|}$     & $\mathcal{|H|}$     & $\mathcal{|A|}$   & $\alpha = 1$ & $\alpha = 0.95$ & $\alpha = 0.90$ &  $\mathcal{|S|}$     & $\mathcal{|H|}$     & $\mathcal{|A|}$   & $\alpha = 1$ & $\alpha = 0.95$ & $\alpha = 0.90$ \\ \hhline{|---|---|---|---|}
	50	&	6	&	3	&	1.04	&	1.47	&	1.83	&	250	&	6	&	3	&	0.41	&	0.98	&	1.00	\\
	50	&	6	&	4	&	2.26	&	4.08	&	3.46	&	250	&	6	&	4	&	2.70	&	2.28	&	2.07	\\
	50	&	6	&	5	&	3.96	&	5.30	&	5.63	&	250	&	6	&	5	&	3.54	&	3.05	&	2.29	\\
	50	&	11	&	3	&	0.57	&	0.99	&	0.60	&	250	&	11	&	3	&	0.27	&	0.59	&	0.80	\\
	50	&	11	&	4	&	$1.22^{\blue{2}}$	&	$2.43^{\blue{2}}$	&	$3.58^{\blue{2}}$	&	250	&	11	&	4	&	$0.82^{\blue{5}}$	&	$0.66^{\blue{5}}$	&	$1.14^{\blue{5}}$	\\
	50	&	11	&	5	&	$2.12^{\blue{5}}$	&	$3.34^{\blue{5}}$	&	$5.79^{\blue{5}}$	&	250	&	11	&	5	&	$1.52^{\blue{5}}$	&	$2.66^{\blue{5}}$	&	$2.03^{\blue{5}}$	\\
	50	&	16	&	3	&	$0.36^{\blue{5}}$	&	$0.70^{\blue{5}}$	&	$0.85^{\blue{5}}$	&	250	&	16	&	3	&	$0.28^{\blue{5}}$	&	$0.44^{\blue{5}}$	&	$0.71^{\blue{5}}$	\\
	50	&	16	&	4	&	$0.94^{\blue{5}}$	&	$1.63^{\blue{5}}$	&	$2.31^{\blue{5}}$	&	250	&	16	&	4	&	$0.55^{\blue{5}}$	&	$0.78^{\blue{5}}$	&	$1.82^{\blue{5}}$	\\
	50	&	16	&	5	&	$1.84^{\blue{5}}$	&	$2.45^{\blue{5}}$	&	$4.25^{\blue{5}}$	&	250	&	16	&	5	&	$1.81^{\blue{5}}$	&	$2.19^{\blue{5}}$	&	$3.02^{\blue{5}}$	\\
	100	&	6	&	3	&	0.67	&	1.48	&	1.86	&	500	&	6	&	3	&	0.41	&	0.98	&	0.94	\\
	100	&	6	&	4	&	3.42	&	3.22	&	2.82	&	500	&	6	&	4	&	1.92	&	1.09	&	0.76	\\
	100	&	6	&	5	&	2.86	&	2.25	&	3.63	&	500	&	6	&	5	&	3.16	&	$2.02^{\blue{2}}$	&	$2.26^{\blue{4}}$	\\
	100	&	11	&	3	&	0.34	&	0.75	&	1.53	&	500	&	11	&	3	&	0.24	&	0.49	&	0.67	\\
	100	&	11	&	4	&	$1.67^{\blue{5}}$	&	$2.56^{\blue{5}}$	&	$2.68^{\blue{5}}$	&	500	&	11	&	4	&	$0.75^{\blue{5}}$	&	$0.88^{\blue{5}}$	&	$1.08^{\blue{5}}$	\\
	100	&	11	&	5	&	$2.25^{\blue{5}}$	&	$4.68^{\blue{5}}$	&	$1.76^{\blue{5}}$	&	500	&	11	&	5	&	$1.20^{\blue{5}}$	&	$2.11^{\blue{5}}$	&	$0.11^{\blue{5}}$	\\
	100	&	16	&	3	&	$0.30^{\blue{5}}$	&	$0.48^{\blue{5}}$	&	$1.09^{\blue{5}}$	&	500	&	16	&	3	&	$0.28^{\blue{5}}$	&	$0.34^{\blue{5}}$	&	$0.50^{\blue{5}}$	\\
	100	&	16	&	4	&	$0.86^{\blue{5}}$	&	$1.33^{\blue{5}}$	&	$3.73^{\blue{5}}$	&	500	&	16	&	4	&	$0.57^{\blue{5}}$	&	$0.64^{\blue{5}}$	&	$-0.98^{\blue{5}}$	\\
	100	&	16	&	5	&	$2.33^{\blue{5}}$	&	$4.14^{\blue{5}}$	&	$3.32^{\blue{5}}$	&	500	&	16	&	5	&	$0.97^{\blue{5}}$	&	$1.11^{\blue{5}}$	&	$-3.31^{\blue{5}}$	\\
 \hhline{|---|---|---|---|}
    \multicolumn{3}{l|}{Average}  & $1.61^{\blue{2.61}}$  & $2.40^{\blue{2.61}}$  & $2.82^{\blue{2.61}}$  &       &       &       & $1.19^{\blue{2.78}}$  & $1.29^{\blue{2.89}}$  & $0.94^{\blue{3}}$ \\
   \multicolumn{3}{l|}{Maximum} & $3.96^{\blue{5}}$  & $5.30^{\blue{5}}$  & $5.79^{\blue{5}}$  &       &       &       & $3.54^{\blue{5}}$  & $3.05^{\blue{5}}$  & $3.02^{\blue{5}}$ \\
    \end{tabular}%
  \label{tab:expATvar}%
\end{table}%

We additionally analyze the optimal policies obtained for the quantile optimization problem QMDP-D \revised{at different confidence levels}, the associated expected value problem \revised{(EV) and the mean value problem (MV)} in comparison to the optimal policies of distinct scenarios. Figure \ref{fig:policy} illustrates the behaviour of the optimal policies under different settings for a particular instance with five scenarios, an inventory capacity of \revised{50} units and at most \revised{four} vehicles, i.e., $|\mathcal{S}|=5$, \revised{$|\mathcal{H}|=6$, $|\mathcal{A}|=5$.} \revised{Figure \ref{fig:policy:b} demonstrates the savings in the quantile function value when the policy minimizing $\alpha$-quantile is used instead of the optimal policies of the expected value problem and the mean value problem at different $\alpha$ values. While the EV policy preserves a consistent gap with the optimal quantile value, the performance of the MV policy fluctuates significantly. Figure \ref{fig:policy:b} also illustrates the nonconvex nature of the quantile function. The performances of the quantile optimizing policy, EV policy and the MV policy converge as $\alpha$ gets closer to 0.80, whereas the performance differences grow for $\alpha$ values larger than 0.80. Additionally, Figure \ref{fig:policy:a} reports the weekly supply and demand rates in each scenario and Figure \ref{fig:policy:c} presents the optimal policies considering these scenarios under different settings.} \revised{In Figure \ref{fig:policy:c}, the dashed lines represent the optimal policies for each scenario independently.} The \revised{solid} line marked with \revised{diamonds} corresponds to the quantile-optimizing policy \revised{at $\alpha=1$}, where \revised{four} vehicles are dispatched whenever the available inventory level drops to \revised{10}, \revised{three} vehicles if the inventory level is in the interval $(10,20]$, and so on. \revised{On the other hand, the policy minimizing $0.60$-quantile, depicted with the solid line marked with squares, is the least conservative: three vehicles are dispatched whenever there is no stocked items, a single vehicle is used for any inventory level in the interval (0,20] and none otherwise.} It can be seen that the optimal policies for scenarios \revised{2, 4 and 5}, which have a \revised{total} probability of 0.60, are similar to each other, while scenarios \revised{1 and 3} also have similar optimal policies \revised{and balanced supply and demand rates}.
As expected, the optimal policy for the quantile optimization problem \revised{at $\alpha=1$} is more aligned with the scenarios \revised{2, 4, and 5}, while \revised{the optimal policy for the quantile optimization problem at $\alpha=0.60$ is closer to the optimal policies of scenarios 1 and 3}. \revised{The EV and the MV policies (depicted with solid lines marked with triangles and circles, respectively) are more moderate and stay in between the quantile-minimizing policies at $\alpha=0.60$ and $\alpha=1$.}

Furthermore, we observe, empirically, that the optimal policies for quantile-optimizing humanitarian inventory management problem are monotone policies, where the number of vehicles to be dispatched decreases as the inventory level increases, as illustrated in Figure \ref{fig:policy:c}. \revised{Note that the provable existence of a monotone optimal policy for a classical MDP is usually contingent upon certain properties of the cost and transition probability functions such as monotonicity and subadditivity (see, e.g., \citealp{puterman2014markov}). However, the nonconvexity of the VaR objective poses significant challenges for proving the existence of optimal monotone policies even when these properties are satisfied. Although it is difficult to make theoretical conclusions about the preservation of monotonicity over a nonconvex quantile  function even if there is a monotone optimal policy for each scenario considered independently, we believe that a monotone policy of ordering nondecreasing quantities of supply as the inventory level decreases is an intuitive (and interpretable) policy that may be appealing to the decision makers in a humanitarian context.} Using such information on the characteristics of the desired policy, it is possible to add the following constraint into our model to enforce a monotone policy structure  
\begin{equation}\label{cons:threshold}
    w_{ia} \leq \sum\limits_{a' \in \mathcal{A}:\ a'\leq a}  w_{i'a'},\quad a \in \mathcal{A},\ i,i' \in \mathcal{H}:\ i'> i.
\end{equation}
Henceforth, we refer to the problem QMDP-D with additional monotonicity constraint \eqref{cons:threshold} as QMDP-M.
Constraint \eqref{cons:threshold} ensures that if $n_a$ vehicles are dispatched at the inventory level $i$, then for any higher inventory level $i'>i$, the number of vehicles dispatched can be at most $n_a$ (assuming $n_a > n_a'$ for $a,a'\in \mathcal{A}$ such that $a>a'$, as given in the problem statement). Incorporating more information on the characteristics of the desired policy may also provide computational advantages as it reduces the solution space.

\begin{figure}

\begin{minipage}[c][.5\linewidth]{.5\linewidth}
\centering
\subfloat[]{\label{fig:policy:b}\includegraphics[width=8.5cm]{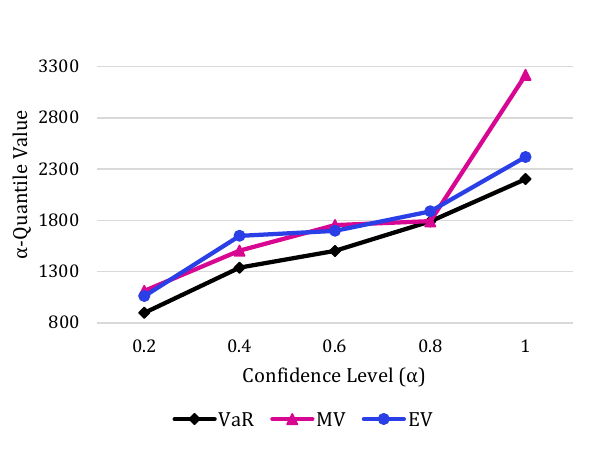}}
\end{minipage}%
\begin{minipage}[c][.5\linewidth]{.5\linewidth}
\centering
\subfloat[]{\begin{tabular}{c|c|c}
    \multicolumn{1}{c|}{Scenario} & \multicolumn{1}{c|}{Demand Rate} & \multicolumn{1}{c}{Supply Rate} \\ \hhline{|-|-|-|}
    1     & 5.22  & 5.67 \\
    2     & 11.71 & 6.60 \\
    3     & 5.07  & 5.11 \\
    4     & 12.19 & 3.78 \\
    5     & 7.88  & 3.13 \\ 
    \end{tabular}\label{fig:policy:a}}
\end{minipage}\par\medskip
\centering
\subfloat[]{\label{fig:policy:c}\includegraphics[width=13.5cm]{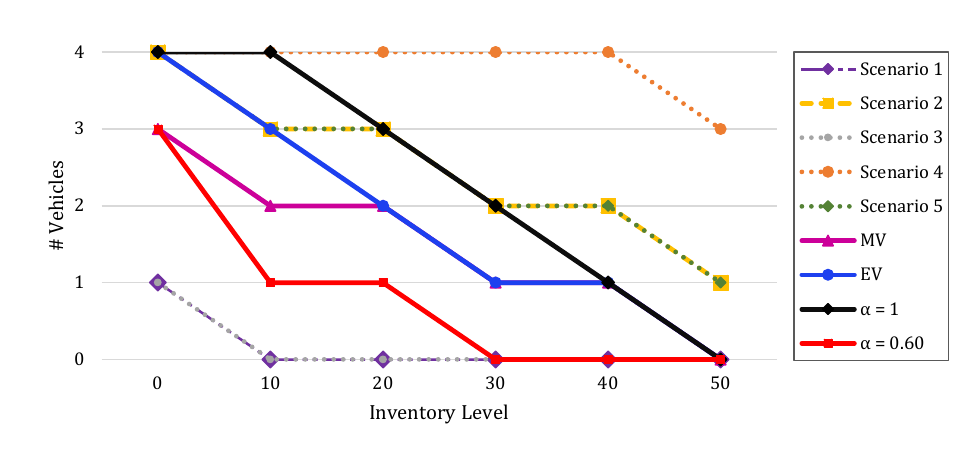}}

\caption{ Comparison of the (a) quantile values for policies minimizing the VaR, expected value and mean value, (b) supply and demand rates for each scenario and (c) policies obtained under different settings.  }
\label{fig:policy}
\end{figure}

\subsection{Comparison of Solution Approaches} \label{sec:CompSolTim}

In this section, we evaluate the computational performance of the proposed heuristic \revised{methods} and the MILP reformulations of QMDP-D in terms of solution times and optimality gaps achieved within the time limit of one hour. Our preliminary experiments using a branch-and-cut algorithm with mixing inequalities indicate poor computational performance because of the difficulty with balancing the strength of bounds obtained from the mixing inequalities with the computational effort required to solve the subproblems in each iteration of the algorithm. 
Hence, the implementation details and results of the branch-and-cut algorithm are omitted here to \revised{conserve} space; \revised{see \citet{merakli2018risk} for more information and our preliminary results on problem instances with fewer scenarios.}

We first examine the \revised{quality}  of the feasible solutions generated by the \revised{proposed} heuristic methods: the mean value (MV) problem, Algorithm \ref{alg2}, and the problem with monotone policies (QMDP-M). \revised{It bears emphasizing that  QMDP-M is considered as a heuristic method because it produces feasible (not necessarily optimal)  solutions to the original problem QMDP-D due to the additional  constraint that restricts the optimal policies to be monotone in the sense that the number of vehicles to be dispatched decreases as the inventory level increases.}  QMDP-M is formulated by adding constraint \eqref{cons:threshold} into the \revised{big-$M$} reformulation of QMDP-D where \revised{constraint \eqref{c8} is replaced by constraints \eqref{c9},} and \revised{the big-$M$ terms in \eqref{ccc3} and \eqref{c9} are set to $M = b_u-b_l$ and $M_{is} = \bar{v}_i^s-\underline{v}_i^s$ for all $i\in \mathcal{H},\ s \in\mathcal{S}$}. Moreover, we add constraints \eqref{cons:lb} \revised{and \eqref{cons:vfBnds},} and apply scenario elimination using the best feasible solution as an upper bound as described in Section \ref{sec:prepro}.
Column ``Time''  of Table \ref{tab:heur} reports the computation time \revised{in seconds} and \revised{column ``\%Diff''  reports the percentage difference with respect to the best solution of QMDP-D } for these three heuristic solution methods \revised{at confidence levels $\alpha = 0.95, 1.00$}. \revised{We additionally report the number of replications that cannot be solved to optimality due to time limit and terminate with a feasible solution using blue superscripts.} The \revised{percentage difference} for a policy with \revised{quantile} function value $obj$ is computed as $100 \times \frac{(obj-obj^*)}{obj}$, where $obj^*$ refers to the \revised{best} objective function value of the quantile minimization problem \revised{obtained within the time limit}. \revised{The results show that the MV problem and Algorithm \ref{alg2} produce feasible policies of similar quality with respect to the average percentage differences. For $\alpha=1.00$, the average percentage difference for Algorithm \ref{alg2} is 1.65\%, while the solutions provided by MV problem are on average within 2.26\% of the best deterministic policy. The maximum differences for both algorithms are 14.40\%. Similarly, for $\alpha=0.95$, the average percentage difference is 3.77\% for the MV solution and 3.90\% for Algorithm \ref{alg2}. However, the MV problem produces feasible policies in at most 15.09 seconds while Algorithm \ref{alg2} may take more than one hour for large instances.} Note that the performance of Algorithm \ref{alg2} could be potentially improved by performing a local search on the scenario decision vector $\vz$. \revised{On the other hand, due to the reduced search space of monotone policies, the heuristic  QMDP-M terminates relatively quickly (in 47.23 and 635.51 seconds for $\alpha=1.00$ and 0.95, respectively), and the monotone policy it generates turns out to be at least as good as the best solution of the original problem for almost all of our instances.}

\begin{table}[h]
  \centering
  \caption{Evaluation of the heuristic algorithms in terms of solution times and optimality gaps}
  \resizebox{16cm}{!}{%
    \begin{tabular}{ccc|rr|rr|rr|rr|rr|rr}
    \multicolumn{3}{r}{}  & \multicolumn{6}{|c}{$\alpha$ = 1.00} & \multicolumn{6}{|c}{$\alpha$  = 0.95} \\ \hhline{|---|--|--|--|--|--|--|}
    \multicolumn{3}{c}{Instance} & \multicolumn{2}{|c|}{MV} & \multicolumn{2}{c|}{Algorithm 1} & \multicolumn{2}{c|}{QMDP-M} & \multicolumn{2}{c|}{MV} & \multicolumn{2}{c|}{Algorithm 1} & \multicolumn{2}{c}{QMDP-M} \\ \hhline{|---|--|--|--|--|--|--|}
   $\mathcal{|S|}$    & $\mathcal{|H|}$    & $\mathcal{|A|}$     & \multicolumn{1}{r}{Time} & \multicolumn{1}{r}{\%Diff} & \multicolumn{1}{|r}{Time} & \multicolumn{1}{r}{\%Diff} & \multicolumn{1}{|r}{Time} & \multicolumn{1}{r}{\%Diff} & \multicolumn{1}{|r}{Time} & \multicolumn{1}{r}{\%Diff} & \multicolumn{1}{|r}{Time} & \multicolumn{1}{r}{\%Diff} & \multicolumn{1}{|r}{Time} & \multicolumn{1}{r}{\%Diff} \\ \hhline{|---|--|--|--|--|--|--|}
50	&	6	&	3	&	0.30	&	0.91	&	13.27	&	0.46	&	0.42	&	0.00	&	0.35	&	4.29	&	15.17	&	0.79	&	0.51	&	0.00	\\
50	&	6	&	4	&	0.35	&	7.66	&	17.46	&	3.72	&	0.61	&	0.00	&	0.40	&	9.87	&	20.18	&	4.13	&	1.37	&	0.00	\\
50	&	6	&	5	&	0.41	&	13.07	&	22.17	&	9.61	&	0.85	&	0.00	&	0.46	&	22.94	&	25.19	&	20.47	&	1.89	&	0.00	\\
50	&	11	&	3	&	0.77	&	0.41	&	36.71	&	0.06	&	1.31	&	0.00	&	0.90	&	0.97	&	42.28	&	0.22	&	1.79	&	0.00	\\
50	&	11	&	4	&	1.04	&	1.43	&	49.44	&	0.41	&	2.42	&	0.00	&	1.05	&	3.27	&	55.41	&	0.92	&	2.90	&	0.00	\\
50	&	11	&	5	&	1.11	&	2.34	&	63.66	&	0.96	&	2.38	&	0.00	&	1.21	&	5.97	&	69.71	&	4.14	&	9.45	&	0.00	\\
50	&	16	&	3	&	1.51	&	0.25	&	72.66	&	0.04	&	2.18	&	0.00	&	1.69	&	0.84	&	82.05	&	0.15	&	3.81	&	0.00	\\
50	&	16	&	4	&	1.75	&	1.04	&	95.09	&	0.16	&	4.47	&	0.00	&	1.98	&	1.78	&	108.25	&	0.34	&	7.19	&	0.00	\\
50	&	16	&	5	&	2.10	&	2.00	&	125.16	&	0.57	&	7.12	&	0.00	&	2.34	&	3.02	&	137.1	&	1.73	&	28.40	&	0.00	\\
100	&	6	&	3	&	0.42	&	0.43	&	26.85	&	0.16	&	0.64	&	0.00	&	0.46	&	0.98	&	32.32	&	0.98	&	1.24	&	0.00	\\
100	&	6	&	4	&	0.48	&	3.93	&	35.90	&	3.93	&	1.28	&	0.00	&	0.52	&	4.72	&	43.14	&	4.72	&	2.61	&	0.00	\\
100	&	6	&	5	&	0.53	&	14.40	&	44.78	&	14.40	&	2.10	&	0.00	&	0.58	&	15.93	&	53.82	&	22.95	&	10.55	&	0.00	\\
100	&	11	&	3	&	1.08	&	0.20	&	75.39	&	0.03	&	1.86	&	0.00	&	1.44	&	0.38	&	86.3	&	0.12	&	5.01	&	0.00	\\
100	&	11	&	4	&	1.25	&	1.72	&	101.98	&	0.47	&	4.03	&	0.00	&	1.36	&	2.60	&	115.39	&	1.37	&	10.68	&	0.00	\\
100	&	11	&	5	&	1.42	&	2.84	&	128.15	&	1.42	&	8.61	&	0.00	&	1.54	&	7.10	&	146.25	&	4.67	&	56.22	&	0.00	\\
100	&	16	&	3	&	2.10	&	0.17	&	151.24	&	0.02	&	4.14	&	0.00	&	2.37	&	0.23	&	175.87	&	0.06	&	9.60	&	0.00	\\
100	&	16	&	4	&	2.43	&	1.01	&	203.29	&	0.15	&	10.21	&	0.00	&	2.72	&	1.72	&	235.28	&	0.74	&	28.98	&	0.00	\\
100	&	16	&	5	&	2.72	&	2.52	&	252.33	&	0.82	&	19.17	&	0.00	&	2.93	&	6.05	&	274.84	&	2.19	&	439.71	&	0.00	\\
250	&	6	&	3	&	0.76	&	0.32	&	68.77	&	0.00	&	0.95	&	0.00	&	0.83	&	0.96	&	94.41	&	0.86	&	4.28	&	0.00	\\
250	&	6	&	4	&	0.82	&	3.42	&	90.81	&	3.42	&	3.17	&	0.00	&	0.88	&	3.41	&	127.28	&	3.41	&	10.97	&	0.00	\\
250	&	6	&	5	&	0.88	&	11.42	&	113.84	&	11.42	&	4.53	&	0.00	&	0.94	&	15.81	&	156.05	&	22.77	&	36.51	&	0.00	\\
250	&	11	&	3	&	2.01	&	0.15	&	193.21	&	0.01	&	5.34	&	0.00	&	2.07	&	0.45	&	241.07	&	0.18	&	16.07	&	0.00	\\
250	&	11	&	4	&	2.16	&	1.03	&	257.62	&	0.32	&	11.78	&	0.00	&	2.24	&	1.28	&	321.58	&	1.16	&	44.86	&	0.00	\\
250	&	11	&	5	&	2.30	&	1.70	&	315.43	&	1.21	&	20.72	&	0.00	&	2.46	&	4.00	&	397.67	&	4.73	&	205.50	&	0.00	\\
250	&	16	&	3	&	3.84	&	0.16	&	374.71	&	0.02	&	10.20	&	0.00	&	4.17	&	0.34	&	427.55	&	0.09	&	34.22	&	0.00	\\
250	&	16	&	4	&	4.34	&	0.52	&	508.76	&	0.17	&	31.86	&	0.00	&	4.22	&	0.99	&	580.6	&	1.12	&	105.47	&	0.00	\\
250	&	16	&	5	&	4.64	&	1.96	&	630.10	&	0.97	&	73.32	&	0.00	&	4.53	&	2.62	&	718.95	&	2.29	&	1176.84	&	0.00	\\
500	&	6	&	3	&	1.34	&	0.31	&	138.24	&	0.06	&	4.20	&	0.00	&	1.51	&	0.65	&	245.97	&	0.65	&	23.23	&	0.00	\\
500	&	6	&	4	&	1.41	&	2.87	&	187.98	&	2.87	&	8.06	&	0.00	&	1.60	&	2.72	&	330.05	&	2.72	&	114.54	&	0.00	\\
500	&	6	&	5	&	1.48	&	7.12	&	234.11	&	7.12	&	8.67	&	0.00	&	1.65	&	14.78	&	413.66	&	29.01	&	1413.45	&	0.00	\\
500	&	11	&	3	&	3.56	&	0.16	&	395.40	&	0.02	&	12.16	&	0.00	&	3.85	&	0.21	&	582.49	&	0.11	&	69.80	&	0.00	\\
500	&	11	&	4	&	3.81	&	1.03	&	542.51	&	0.35	&	24.67	&	0.00	&	3.89	&	1.72	&	741.39	&	1.72	&	233.66	&	0.00	\\
500	&	11	&	5	&	3.78	&	1.43	&	647.11	&	0.88	&	47.36	&	0.00	&	4.33	&	3.74	&	1028.56	&	3.82	&	$3221.29^{\blue{4}}$	&	0.00	\\
500	&	16	&	3	&	6.82	&	0.16	&	768.25	&	0.02	&	31.17	&	0.00	&	7.55	&	0.21	&	1092.92	&	0.07	&	255.53	&	0.00	\\
500	&	16	&	4	&	7.20	&	0.49	&	1046.28	&	0.15	&	75.58	&	0.00	&	7.83	&	1.13	&	1448.83	&	0.95	&	881.26	&	0.00	\\
500	&	16	&	5	&	7.58	&	0.91	&	1277.61	&	0.52	&	238.15	&	0.00	&	7.89	&	1.51	&	1723.18	&	1.65	&	$3046.32^{\blue{4}}$	&	0.00	\\
1000	&	6	&	3	&	2.20	&	0.26	&	243.07	&	0.22	&	8.32	&	0.00	&	2.80	&	0.45	&	707.36	&	0.71	&	93.71	&	0.00	\\
1000	&	6	&	4	&	2.28	&	1.89	&	326.49	&	1.89	&	20.36	&	0.00	&	2.90	&	2.44	&	945.14	&	2.44	&	603.58	&	0.00	\\
1000	&	6	&	5	&	2.47	&	3.12	&	422.04	&	3.12	&	18.71	&	0.00	&	2.69	&	13.44	&	1109.15	&	20.43	&	2112.05	&	0.00	\\
1000	&	11	&	3	&	6.40	&	0.14	&	768.36	&	0.06	&	35.94	&	0.00	&	7.42	&	0.20	&	1643.97	&	0.20	&	306.78	&	0.00	\\
1000	&	11	&	4	&	7.18	&	0.93	&	1117.65	&	0.35	&	122.17	&	0.00	&	6.64	&	1.17	&	1950.06	&	1.17	&	$2586.68^{\blue{1}}$	&	0.00	\\
1000	&	11	&	5	&	7.17	&	1.31	&	1371.73	&	0.69	&	168.70	&	0.00	&	6.97	&	2.45	&	2455.05	&	2.63	&	$3600.56^{\blue{5}}$	&	0.00	\\
1000	&	16	&	3	&	12.68	&	0.15	&	1550.63	&	0.04	&	218.16	&	0.00	&	11.08	&	0.26	&	2376.31	&	0.09	&	710.92	&	0.00	\\
1000	&	16	&	4	&	13.37	&	1.06	&	2080.92	&	0.31	&	378.42	&	0.00	&	13.85	&	0.13	&	$3607.18^{\blue{1}}$	&	0.24	&	$3467.09^{\blue{4}}$	&	0.20	\\
1000	&	16	&	5	&	15.09	&	1.36	&	2876.66	&	0.59	&	469.08	&	0.00	&	12.61	&	-0.07	&	$3912.40^{\blue{2}}$	&	-0.45	&	$3601.10^{\blue{5}}$	&	0.00	\\
 \hhline{|---|--|--|--|--|--|--|}
    \multicolumn{3}{l|}{Average}  & 3.32  & 2.26  & 445.86 & 1.65  & 47.23 & 0.00  & 3.42  & 3.77  & $691.05^{\blue{0.07}}$ & 3.90  & $635.51^{\blue{0.51}}$ & 0.00 \\
     \multicolumn{3}{l|}{Maximum}  & 15.09 & 14.40 & 2876.66 & 14.40 & 469.08 & 0.00  & 13.85 & 22.94 & $3912.40^{\blue{2}}$ & 29.01 & $3601.10^{\blue{5}}$ & 0.20 \\
    \end{tabular}%
    }
  \label{tab:heur}%
\end{table}%


Next we evaluate the computational performance of the proposed MILP reformulations of QMDP-D \revised{at confidence level $\alpha=0.95$} under the settings detailed below. Note that these MILP reformulations provide an exact representation of problem QMDP-D.
\begin{itemize}
    \item[-] \revised{\emph{BM -- Basic}: It corresponds to formulation QMDP-D where constraint \eqref{c8} is replaced by constraint \eqref{c9} with additional big-$M$ terms. All big-$M$ terms are fixed as $10^6$.}
    \item[-] \emph{BM}: \revised{In addition to the setting in \emph{BM -- Basic},} the big-$M$ term in \eqref{ccc3} is set to \revised{$M = b_u-b_l$} and the big-$M$ terms $M_{is} = \bar{v}_i^s-\underline{v}_i^s, \ i\in \mathcal{H},s\in\mathcal{S}$ in constraint \eqref{c9}. Moreover, constraints \eqref{cons:lb} \revised{and \eqref{cons:vfBnds} are} added and the scenario elimination procedure described in Section \ref{sec:prepro} is applied using the optimal monotone policy solution of QMDP-M as an upper bound. Gurobi solver is provided with the optimal monotone policy solution as an initial feasible solution. 
    \item[-] \revised{\emph{MC -- Basic}:  It corresponds to the McCormick reformulation of QMDP-D with constraints \eqref{form:mc1} and \eqref{form:mc2} replacing \eqref{c8}. All big-$M$ terms are fixed as $10^6$ and lower bounds are set to zero.}
   \item[-] \emph{MC}: The big-$M$ term in \eqref{ccc3} is set to \revised{$M = b_u-b_l$ in formulation \emph{MC -- Basic}}. \revised{Additionally, the lower and upper bounds on variable $v_i^s$ in \eqref{form:mc1} are $\ell_j^s=\underline{v}_i^s$ and $u_j^s=\bar{v}_i^s$ for $i\in \mathcal{H},s\in\mathcal{S}$, respectively.} The optimal monotone policy solution  of QMDP-M is embedded into the solver as an initial feasible solution. Constraint \eqref{cons:lb} is added and scenario elimination is applied using the optimal monotone policy solution as an upper bound as described in Section \ref{sec:prepro}.
\end{itemize}


In Table \ref{tab:compDEF}, we report the total solution time in seconds (``Time (s)''), the best optimality gap achieved within the time limit (``Gap (\%)''), and the number of nodes in the branch-and-bound tree (``Nodes'') for the MILP reformulations detailed above. The optimality gap values are computed as $100\times\frac{ub-lb}{ub}$, where $ub$ and $lb$ correspond to the best upper and lower bounds on the optimal objective function value achieved at termination, respectively. The reported solution times do not include the time to obtain bounds and initial solutions. Each row corresponds to the average of \revised{five} replications and \revised{each blue (resp., red) superscript in columns ``Time (s)'' indicates the number of replications that exceeded the time limit with (resp., without) a feasible solution.} 
The results in Table \ref{tab:compDEF} show that \revised{the MILP models with additional big-$M$ terms (\emph{BM} and \emph{BM -- Basic}) outperform the corresponding MILP models with McCormick reformulation (\emph{MC} and \emph{MC -- Basic}). Out of 225 instances \emph{BM -- Basic} can solve 60 instances to optimality within the time limit, whereas \emph{MC -- Basic} can achieve optimality within the time limit for only 11 instances and fails to obtain a feasible solution for six instances. Our results also demonstrate the effectiveness of preprocessing methods for both formulations. The solution times are reduced by 48.79\% and 24.87\% for \emph{BM} and \emph{MC}, respectively. Out of 225, the number of instances that can be solved to optimality increases from 60 instances to 146 for \emph{BM} and from 11 instances to 76 for \emph{MC}. Similarly, the percentage optimality gaps and the number of nodes in the branch-and-bound tree usually decrease after the preprocessing methods are applied. Overall, formulation \emph{BM} with preprocessing demonstrates the best performance with respect to both the solution times and the number of instances that can be solved to optimality.} 


\begin{landscape}

\begin{table}[htbp]
  \centering
  \caption{Comparison of MILP reformulations of QMDP-D}
  \resizebox{21cm}{!}{%
    \begin{tabular}{ccc|rrr|rrr|rrr|rrr}
    \multicolumn{3}{c|}{Instance}      & \multicolumn{3}{|c}{\emph{BM}}           & \multicolumn{3}{|c}{\emph{BM -- basic}}     & \multicolumn{3}{|c}{\emph{MC}}   & \multicolumn{3}{|c}{\emph{MC -- basic}}     \\ \hhline{|---|---|---|---|---|}
    $\mathcal{|S|}$     & $\mathcal{|H|}$     & $\mathcal{|A|}$      & Time (s) & Gap (\%) & Nodes  &  Time (s) & Gap (\%) & Nodes & Time (s) & Gap (\%) & Nodes    & Time (s) & Gap (\%) & Nodes  \\ \hhline{|---|---|---|---|---|}
50	&	6	&	3	&	0.53	&	0.00	&	74.0	&	3.71	&	0.00	&	1576.4	&	8.81	&	0.00	&	94.40	&	68.97	&	0.00	&	2524.0	\\
50	&	6	&	4	&	1.54	&	0.00	&	477.6	&	123.76	&	0.00	&	9177.4	&	39.69	&	0.00	&	889.80	&	1284.51	&	0.00	&	17011.8	\\
50	&	6	&	5	&	6.96	&	0.00	&	2702.2	&	188.00	&	0.00	&	15325.0	&	340.62	&	0.00	&	7263.60	&	$3569.76^{\blue{4}}$	&	74.79	&	28518.6	\\
50	&	11	&	3	&	5.01	&	0.00	&	1115.4	&	595.36	&	0.00	&	39926.8	&	151.53	&	0.00	&	558.80	&	$3600.01^{\blue{5}}$	&	94.17	&	9179.4	\\
50	&	11	&	4	&	48.97	&	0.00	&	9489.2	&	$3482.07^{\blue{4}}$	&	62.91	&	260029.0	&	$1605.48^{\blue{1}}$	&	1.09	&	4852.20	&	$3600.01^{\blue{5}}$	&	98.60	&	5664.2	\\
50	&	11	&	5	&	1371.54	&	0.00	&	180041.6	&	$3600.00^{\blue{5}}$	&	90.91	&	134324.6	&	$2930.94^{\blue{4}}$	&	6.19	&	9891.00	&	$3600.01^{\blue{5}}$	&	98.52	&	5257.2	\\
50	&	16	&	3	&	48.41	&	0.00	&	4955.2	&	$3600.00^{\blue{5}}$	&	47.15	&	116557.8	&	$996.60^{\blue{1}}$	&	0.15	&	1100.00	&	$3600.06^{\blue{5}}$	&	99.16	&	1449.0	\\
50	&	16	&	4	&	$773.00^{\blue{1}}$	&	0.26	&	30838.0	&	$3600.00^{\blue{5}}$	&	75.61	&	74818.2	&	$2264.10^{\blue{3}}$	&	0.55	&	1815.20	&	$3600.04^{\blue{5}}$	&	99.17	&	1281.2	\\
50	&	16	&	5	&	$2883.24^{\blue{4}}$	&	2.72	&	66421.6	&	$3600.00^{\blue{5}}$	&	82.87	&	23468.2	&	$3077.71^{\blue{4}}$	&	3.03	&	2822.00	&	$3600.04^{\blue{5}}$	&	99.16	&	1041.4	\\
100	&	6	&	3	&	1.11	&	0.00	&	60.0	&	14.21	&	0.00	&	2597.4	&	22.18	&	0.00	&	154.80	&	$2260.04^{\blue{3}}$	&	50.70	&	9626.6	\\
100	&	6	&	4	&	6.76	&	0.00	&	1284.8	&	348.74	&	0.00	&	13624.4	&	338.6	&	0.00	&	3089.20	&	$3600.04^{\blue{5}}$	&	98.37	&	8594.2	\\
100	&	6	&	5	&	109.59	&	0.00	&	9182.8	&	949.59	&	0.00	&	32120.0	&	$3300.32^{\blue{4}}$	&	11.39	&	11485.00	&	$3600.01^{\blue{5}}$	&	98.60	&	6170.4	\\
100	&	11	&	3	&	1.75	&	0.00	&	23.2	&	$2554.87^{\blue{1}}$	&	13.88	&	68821.4	&	279.21	&	0.00	&	361.20	&	$3600.03^{\blue{5}}$	&	99.00	&	2259.2	\\
100	&	11	&	4	&	331.41	&	0.00	&	14263.8	&	$3600.00^{\blue{5}}$	&	91.92	&	40103.6	&	$3102.97^{\blue{4}}$	&	2.00	&	4304.20	&	$3600.03^{\blue{5}}$	&	98.95	&	1000.4	\\
100	&	11	&	5	&	$3483.97^{\blue{4}}$	&	7.61	&	153772.0	&	$3600.00^{\blue{5}}$	&	97.98	&	30367.0	&	$3600.02^{\blue{5}}$	&	10.44	&	3403.60	&	$3600.04^{\blue{5}}$	&	99.07	&	720.0	\\
100	&	16	&	3	&	4.07	&	0.00	&	7.0	&	$3600.02^{\blue{5}}$	&	69.47	&	21071.6	&	263.18	&	0.00	&	33.00	&	$3600.05^{\blue{5}}$	&	99.61	&	292.8	\\
100	&	16	&	4	&	$1762.18^{\blue{2}}$	&	1.04	&	26697.6	&	$3600.03^{\blue{5}}$	&	89.07	&	6343.2	&	$3600.02^{\blue{5}}$	&	5.03	&	1107.20	&	$3600.04^{\blue{5}}$	&	99.49	&	261.0	\\
100	&	16	&	5	&	$3600.01^{\blue{5}}$	&	5.07	&	38582.4	&	$3600.12^{\blue{5}}$	&	98.09	&	9030.0	&	$3600.05^{\blue{5}}$	&	10.41	&	1014.60	&	$3600.08^{\blue{5}}$	&	99.30	&	309.4	\\
250	&	6	&	3	&	3.40	&	0.00	&	97.6	&	307.79	&	0.00	&	5921.6	&	88.21	&	0.00	&	142.60	&	$3600.03^{\blue{5}}$	&	98.91	&	3180.2	\\
250	&	6	&	4	&	26.24	&	0.00	&	2023.4	&	1692.88	&	0.00	&	16706.6	&	$1890.50^{\blue{1}}$	&	0.91	&	3756.60	&	$3600.03^{\blue{5}}$	&	99.14	&	1493.8	\\
250	&	6	&	5	&	317.29	&	0.00	&	8098.8	&	$3233.03^{\blue{2}}$	&	34.70	&	38015.6	&	$3600.11^{\blue{5}}$	&	10.81	&	2443.00	&	$3600.05^{\blue{5}}$	&	99.13	&	1046.4	\\
250	&	11	&	3	&	13.97	&	0.00	&	154.8	&	$3600.00^{\blue{5}}$	&	93.90	&	17657.2	&	$1108.67^{\blue{1}}$	&	0.10	&	312.80	&	$3600.55^{\blue{5}}$	&	99.48	&	307.0	\\
250	&	11	&	4	&	$1558.46^{\blue{1}}$	&	0.49	&	13295.4	&	$3600.04^{\blue{5}}$	&	99.17	&	8746.2	&	$3600.16^{\blue{5}}$	&	7.79	&	698.80	&	$3600.06^{\blue{5}}$	&	99.47	&	238.4	\\
250	&	11	&	5	&	$3600.03^{\blue{5}}$	&	5.19	&	28274.2	&	$3600.06^{\blue{5}}$	&	99.60	&	2696.2	&	$3600.10^{\blue{5}}$	&	10.92	&	334.00	&	$3600.11^{\blue{5}}$	&	99.48	&	125.4	\\
250	&	16	&	3	&	42.39	&	0.01	&	602.0	&	$3600.02^{\blue{5}}$	&	97.76	&	4886.2	&	$3117.60^{\blue{3}}$	&	0.06	&	316.00	&	$3600.12^{\blue{5}}$	&	99.71	&	81.2	\\
250	&	16	&	4	&	$3035.12^{\blue{4}}$	&	1.42	&	12251.6	&	$3600.05^{\blue{5}}$	&	99.67	&	1665.0	&	$3600.30^{\blue{5}}$	&	20.99	&	134.00	&	$3600.27^{\blue{5}}$	&	99.68	&	47.2	\\
250	&	16	&	5	&	$3600.01^{\blue{5}}$	&	5.34	&	10113.2	&	$3600.03^{\blue{5}}$	&	99.93	&	4998.0	&	$3600.32^{\blue{5}}$	&	15.47	&	37.80	&	$3600.26^{\blue{5}}$	&	99.59	&	43.0	\\
500	&	6	&	3	&	10.04	&	0.00	&	87.8	&	1219.99	&	0.00	&	8865.8	&	$2299.07^{\blue{1}}$	&	1.00	&	326.20	&	$3601.08^{\blue{5}}$	&	99.37	&	728.2	\\
500	&	6	&	4	&	228.39	&	0.00	&	3696.8	&	$3479.89^{\blue{4}}$	&	78.56	&	10500.6	&	$3600.41^{\blue{5}}$	&	41.51	&	0.80	&	$3600.97^{\blue{5}}$	&	99.40	&	486.8	\\
500	&	6	&	5	&	$3137.55^{\blue{2}}$	&	2.42	&	25889.2	&	$3600.01^{\blue{5}}$	&	99.27	&	8176.0	&	$3601.14^{\blue{5}}$	&	68.80	&	0.60	&	$3600.85^{\blue{5}}$	&	99.46	&	462.4	\\
500	&	11	&	3	&	30.83	&	0.00	&	116.2	&	$3600.01^{\blue{5}}$	&	99.16	&	5640.8	&	$2659.70^{\blue{2}}$	&	0.06	&	114.00	&	$3600.05^{\blue{5}}$	&	99.59	&	140.8	\\
500	&	11	&	4	&	$3077.46^{\blue{3}}$	&	0.74	&	7469.6	&	$3600.02^{\blue{5}}$	&	99.69	&	1658.4	&	$3600.25^{\blue{5}}$	&	19.43	&	1.00	&	$3600.10^{\blue{5}}$	&	99.56	&	72.4	\\
500	&	11	&	5	&	$3600.02^{\blue{5}}$	&	7.48	&	4394.4	&	$3600.27^{\blue{5}}$	&	99.75	&	1527.2	&	$3600.24^{\blue{5}}$	&	28.53	&	1.00	&	$3600.13^{\blue{5}}$	&	99.57	&	62.8	\\
500	&	16	&	3	&	264.96	&	0.00	&	1215.0	&	$3600.01^{\blue{5}}$	&	99.63	&	1867.8	&	$3600.14^{\blue{5}}$	&	0.27	&	0.00	&	$3600.26^{\blue{4},\red{1}}$	&	99.75	&	10.5	\\
500	&	16	&	4	&	$3600.02^{\blue{5}}$	&	1.36	&	2104.8	&	$3600.03^{\blue{5}}$	&	99.88	&	1170.6	&	$3600.14^{\blue{5}}$	&	37.67	&	0.20	&	$3600.14^{\blue{4},\red{1}}$	&	99.72	&	10.3	\\
500	&	16	&	5	&	$3600.04^{\blue{5}}$	&	6.72	&	1725.2	&	$3600.02^{\blue{5}}$	&	100.00	&	1702.8	&	$3600.16^{\blue{5}}$	&	81.72	&	0.00	&	$3600.64^{\blue{5}}$	&	99.77	&	6.0	\\
1000	&	6	&	3	&	74.47	&	0.00	&	638.0	&	$3472.34^{\blue{4}}$	&	78.97	&	6475.6	&	$3600.08^{\blue{5}}$	&	6.16	&	1.00	&	$3600.10^{\blue{5}}$	&	99.48	&	364.0	\\
1000	&	6	&	4	&	$1930.16^{\blue{1}}$	&	0.50	&	6128.6	&	$3600.07^{\blue{5}}$	&	99.37	&	4296.4	&	$3600.16^{\blue{5}}$	&	17.86	&	1.00	&	$3600.05^{\blue{5}}$	&	99.55	&	251.8	\\
1000	&	6	&	5	&	$3600.07^{\blue{5}}$	&	8.44	&	11428.0	&	$3600.39^{\blue{5}}$	&	99.44	&	3478.2	&	$3600.12^{\blue{5}}$	&	43.53	&	1.00	&	$3600.04^{\blue{5}}$	&	99.53	&	170.0	\\
1000	&	11	&	3	&	210.63	&	0.00	&	563.4	&	$3600.09^{\blue{5}}$	&	99.48	&	2491.0	&	$3600.18^{\blue{5}}$	&	14.97	&	0.00	&	$3600.60^{\blue{5}}$	&	99.75	&	1.0	\\
1000	&	11	&	4	&	$3600.03^{\blue{5}}$	&	3.07	&	1293.8	&	$3600.32^{\blue{5}}$	&	99.81	&	1444.0	&	$3600.18^{\blue{5}}$	&	81.88	&	0.00	&	$3600.39^{\blue{2},\red{3}}$	&	99.86	&	1.0	\\
1000	&	11	&	5	&	$3600.02^{\blue{5}}$	&	6.95	&	1353.0	&	$3600.43^{\blue{5}}$	&	99.89	&	1036.8	&	$3600.16^{\blue{5}}$	&	100.00	&	0.00	&	$3600.28^{\blue{5}}$	&	99.90	&	1.0	\\
1000	&	16	&	3	&	$2003.15^{\blue{2}}$	&	0.24	&	1155.4	&	$3600.23^{\blue{5}}$	&	99.82	&	977.0	&	$3600.21^{\blue{5}}$	&	28.95	&	0.00	&	$3600.43^{\blue{5}}$	&	99.82	&	1.0	\\
1000	&	16	&	4	&	$3600.09^{\blue{5}}$	&	3.58	&	461.2	&	$3600.04^{\blue{5}}$	&	99.97	&	809.0	&	$3600.41^{\blue{5}}$	&	85.56	&	0.00	&	$3600.62^{\blue{4},\red{1}}$	&	99.80	&	1.0	\\
1000	&	16	&	5	&	$3600.11^{\blue{5}}$	&	6.30	&	533.0	&	$3600.04^{\blue{5}}$	&	99.99	&	696.6	&	$3600.45^{\blue{5}}$	&	100.00	&	0.00	&	$3601.29^{\blue{5}}$	&	99.96	&	1.0	\\
\hhline{|---|---|---|---|---|}
    \multicolumn{3}{l|}{Average } &	$1475.67^{\blue{1.76}}$	&	1.71	&	15225.6	&	$2881.52^{\blue{3.67}}$	&	68.83	&	23630.9	&	$2584.25^{\blue{3.31}}$	&	19.45	&	1396.9	&	$3439.85^{\blue{4.57},\red{0.13}}$	&	93.22	&	2455.5	\\
    \multicolumn{3}{l|}{Maximum} &	$3600.11^{\blue{5}}$	&	8.44	&	180041.6	&	$3600.43^{\blue{5}}$	&	100.00	&	260029.0	&	$3601.14^{\blue{5}}$	&	100.00	&	11485.0	&	$3601.29^{\blue{5},\red{3}}$	&	99.96	&	28518.6
 \\
     \end{tabular}%
    }
  \label{tab:compDEF}%
\end{table}%
\end{landscape}

\section{Conclusions} \label{sec:Conc}

In this study, we investigate the risk associated with parameter uncertainty in MDPs. We formulate the problem with the objective of minimizing the VaR of the expected total discounted cost of an MDP at a prespecified confidence level $\alpha$ and explore characteristics of the optimal policies. Assuming a discrete representation of uncertainty, we provide MINLP and MILP formulations considering randomized and deterministic policies, and propose preprocessing methods and heuristic algorithms that can be applied for both cases. The proposed modeling approach and solution algorithms are tested on an inventory management problem in the long term humanitarian relief operations context \revised{and our results show that the proposed modeling approach may provide significant reductions in risk arising from parameter uncertainty compared with solving an MDP with average system parameters or solving an expected value problem as proposed in \cite{steimle2018multimodel}.}

\revised{In contrast to classical MDPs in which system parameters are assumed to be known, all optimal policies in an MDP with uncertain parameters minimizing the associated VaR may be randomized. The results on our instances of the humanitarian inventory management problem indicate that policy randomization may provide additional savings in the expected discounted total cost, but it also amplifies computational complexity of the problem due to additional nonlinearities. Hence, we recommend the use of deterministic policies especially in humanitarian settings, where implementation of randomized policies may raise ethical concerns. On the other hand, our mixed-integer nonlinear programming formulation QMDP-R can be used for application areas such as power management, infrastructure maintenance and rehabilitation and queuing systems, where randomized policies are effectively implemented.} 

\revised{Comparing the solution methods for the deterministic policy case, formulation $BM$ with additional big-$M$ terms outperforms the formulations based on McCormick envelopes. Our preprocessing procedures also provide significant improvements on the computational performance of both formulations. In addition, we provide heuristic methods and show that QMDP-M, enforcing monotonicity constraints on the optimal policy, performs best since it reduces solution times remarkably and provides feasible solutions that are also optimal for the original problem on our problem instances. }

\revised{It is worthy to note that, just as policy randomization, it may also be possible to gain from utilizing nonstationary policies in a setting where the uncertain parameters are learned over time. 
However, computing a nonstationary optimal policy that minimizes the risk arising from parameter uncertainty in an infinite-horizon setting is possible if the objective is  to optimize a dynamic risk measure. On the other hand, the risk measure of interest in our paper,  namely VaR,  is not a dynamic risk measure and hence it is not amenable to dynamic updates.  \cite{zhang2014branch} discuss the challenges in using chance constraints in a dynamic framework and show that solving static optimization problems at each period in a rolling horizon fashion may result in violations in the chance constraint representing a service level requirement of the performance of the system over the entire planning horizon. Moreover, \cite{steimle2018multimodel} demonstrate that learning the uncertain parameters over time significantly increases the computational complexity even in the case of the expected value criterion while providing low level of gains in the objective function. Hence, we believe that the use of dynamic risk measures for addressing parameter uncertainty in MDPs would be an interesting and nontrivial extension of this study. We  refer interested readers to a recent paper by \cite{dentcheva2018risk}, which provides a foundation for dynamic risk measures under distributional uncertainty. }

  \section*{Acknowledgment}
We thank the AE and the three anonymous referees whose comments improved the paper.

{\singlespace
\bibliographystyle{apa}
\bibliography{mybibfile}}

\end{document}